\documentclass[12pt]{amsart}
\usepackage[T1]{fontenc}
\usepackage[utf8]{inputenc}
\usepackage{lmodern}
\usepackage{amsmath,amssymb,mathtools}
\usepackage{aliascnt}
\usepackage{microtype}
\usepackage[a4paper,margin=30mm]{geometry}
\usepackage{xcolor}
\usepackage[
  colorlinks=true,
  linkcolor=blue,
  citecolor=blue,
  urlcolor=blue,
  filecolor=blue,
  linktoc=all,
  bookmarksopen=true,
  bookmarksnumbered=true,
  pdfpagemode=UseOutlines
]{hyperref}
\usepackage{bookmark}

\allowdisplaybreaks
\newtheorem{theorem}{Theorem}[section]
\newaliascnt{proposition}{theorem}
\newtheorem{proposition}[proposition]{Proposition}
\aliascntresetthe{proposition}
\newaliascnt{corollary}{theorem}
\newtheorem{corollary}[corollary]{Corollary}
\aliascntresetthe{corollary}
\newaliascnt{lemma}{theorem}
\newtheorem{lemma}[lemma]{Lemma}
\aliascntresetthe{lemma}
\newtheorem{maintheorem}{Theorem}

\theoremstyle{remark}
\newaliascnt{remark}{theorem}
\newtheorem{remark}[remark]{Remark}
\aliascntresetthe{remark}
\newaliascnt{example}{theorem}
\newtheorem{example}[example]{Example}
\aliascntresetthe{example}

\usepackage[noabbrev]{cleveref}

\crefname{theorem}{Theorem}{Theorems}
\Crefname{theorem}{Theorem}{Theorems}
\crefname{maintheorem}{Theorem}{Theorems}
\Crefname{maintheorem}{Theorem}{Theorems}
\crefname{proposition}{Proposition}{Propositions}
\Crefname{proposition}{Proposition}{Propositions}
\crefname{corollary}{Corollary}{Corollaries}
\Crefname{corollary}{Corollary}{Corollaries}
\crefname{lemma}{Lemma}{Lemmas}
\Crefname{lemma}{Lemma}{Lemmas}
\crefname{remark}{Remark}{Remarks}
\Crefname{remark}{Remark}{Remarks}
\crefname{example}{Example}{Examples}
\Crefname{example}{Example}{Examples}
\crefname{section}{Section}{Sections}
\Crefname{section}{Section}{Sections}
\crefname{subsection}{Section}{Sections}
\Crefname{subsection}{Section}{Sections}
\crefname{equation}{Equation}{Equations}
\Crefname{equation}{Equation}{Equations}

\newcommand{\C}{\mathbb C}
\newcommand{\Q}{\mathbb Q}
\newcommand{\Z}{\mathbb Z}
\newcommand{\bx}{\mathbf{x}}
\newcommand{\by}{\mathbf{y}}
\newcommand{\bz}{\mathbf{z}}
\newcommand{\Sp}{\operatorname{Sp}}
\newcommand{\Gr}{\operatorname{Gr}}
\newcommand{\im}{\operatorname{im}}
\newcommand{\coker}{\operatorname{coker}}
\newcommand{\Tj}{\mathrm{Tj}}
\newcommand{\St}{\mathrm{St}}
\newcommand{\HI}{\mathrm{HI}}
\newcommand{\eps}{\varepsilon}
\newcommand{\defect}{\mathcal E}
\newcommand{\level}{\operatorname{level}}
\newcommand{\boxplusop}{\mathbin{\boxplus}}
\newcommand{\ind}{\mathbf 1}

\title[Hodge--Tjurina numbers and Thom--Sebastiani formulas]
{Hodge--Tjurina Numbers, Tjurina Spectra, and Thom--Sebastiani Formulas}
\author{Shijie Zheng}
\address{Department of Mathematical Sciences, Tsinghua University, Beijing 100084, P. R. China}
\email{zheng-sj23@mails.tsinghua.edu.cn}
\author{Huaiqing Zuo}
\address{Department of Mathematical Sciences, Tsinghua University, Beijing 100084, P. R. China}
\email{hqzuo@mail.tsinghua.edu.cn}
\thanks{H. Zuo acknowledges support from NSFC (grant No. 12671056) and BJNSF (grant No. 1252009).}

\begin{document}
\begin{abstract}
We study Hodge--Tjurina numbers and the Tjurina subspectrum of isolated hypersurface singularities. We relate these invariants to Hodge ideals and microlocal $V$-filtrations and investigate their behavior under Thom--Sebastiani sums. For separated sums, we obtain a filtered defect formula and, in particular, give an affirmative answer to the question raised by Almir\'on concerning the Thom--Sebastiani property of the Tjurina subspectrum: it holds if and only if the Tjurina number is multiplicative. We also discuss rational sampling, spectral reconstruction, symmetry bounds, and examples illustrating the information captured and lost by integer Hodge--Tjurina data.
\end{abstract}

\subjclass[2020]{Primary 14B05; Secondary 32S05, 32S35}
\keywords{Hodge ideals, Tjurina algebra, microlocal $V$-filtration, Tjurina subspectrum, Thom--Sebastiani formula}
\maketitle
\tableofcontents

\section{Introduction}

Hodge ideals and the spectrum provide two important Hodge-theoretic invariants. Musta\c{t}\u{a}--Popa introduced Hodge ideals for reduced divisors \cite{MP-HI} and extended them to effective $\Q$-divisors \cite{MP-QBir}. Hodge ideals extend multiplier ideals and measure the singularities of a divisor through the filtration.

Let $R=\C\{x_1,\ldots,x_n\}$ with $n\ge2$ denote the ring of convergent power series in $n$ variables, and let $f\in R$ be the germ of a holomorphic function defining an isolated hypersurface singularity. Let $Z_f=(f=0)$. For $\lambda\in\Q\cap(0,1]$, we denote by $I_p(\lambda Z_f)$ the $p$-th Hodge ideal of the effective $\Q$-divisor $\lambda Z_f$. The precise $\Q$-divisor definition used here is recalled in \cref{sec:JKSY}.

The spectrum is another fundamental invariant of an isolated hypersurface singularity. This invariant was first defined by Steenbrink; see \cite{Steenbrink77,Steenbrink89} for details. Let $J_f=\left(\frac{\partial f}{\partial x_1},\ldots,\frac{\partial f}{\partial x_n}\right)$
denote the Jacobian ideal. We call $A_f:=R/J_f$ the Milnor algebra and $T_f:=R/(f,J_f)$ the Tjurina algebra of $f$. Their dimensions $\mu_f:=\dim_{\C}A_f$ and $\tau_f:=\dim_{\C}T_f$
are the Milnor number and Tjurina number of $f$, respectively. We will review these algebraic preliminaries in \cref{sec:prelim}.

Then the spectrum of $f$ is $\Sp_f(t)=\sum_{i=1}^{\mu_f} t^{\alpha^{\St}_{f,i}},$ where the positive rational numbers $\alpha^{\St}_{f,i}$ are the spectral numbers (or exponents). See \cref{sec:JKSY}, especially \eqref{eq:St-mult}, for precise definitions and details. In Steenbrink's original definition, these rational numbers come from the mixed Hodge structure on the vanishing cohomology; see \cite{Steenbrink77,Steenbrink89} for details. Jung--Kim--Saito--Yoon (JKSY) \cite{JKSY} used Hodge ideals together with the microlocal $V$-filtration to define the Hodge-ideal spectrum and the Tjurina subspectrum, and compared them with Steenbrink's spectrum. Here we write $\Sp_f^{\Tj}(t)=\sum_{i=1}^{\tau_f}t^{\alpha^{\Tj}_{f,i}}$ for the JKSY Tjurina subspectrum. The subspectrum is obtained by passing through the Hodge-ideal filtration to the Tjurina algebra; see \cite{JKSY}. See \cref{sec:JKSY}, especially \eqref{eq:Tj-mult}, for precise definitions.

In this paper, we focus on the rationally indexed Hodge--Tjurina quotient. For $c\in\Q_{>0}$, write uniquely $c=p+\lambda$ with $p\in\Z_{\ge0}$ and $\lambda\in\Q\cap(0,1]$, and define
\begin{equation}\label{eq:Hdef}
 H_f(c):=\dim_{\C}\frac{R}{I_p(\lambda Z_f)+(f,J_f)}.
\end{equation}
We call $H_f(c)$ the rational Hodge--Tjurina number at $c$, and set $H_f(c)=0$ for $c\le0$. For the integral divisor $D=Z_f$, we write
\[
 m_k(f):=H_f(k+1)=\dim_{\C}\frac{R}{I_k(Z_f)+(f,J_f)}.
\]

This construction is closely related to the Hodge moduli algebras studied by Wang--Yau--Zuo \cite[Definition~2.2]{WYZ}. For $D_\lambda=\lambda Z_f$ and $p\in\Z_{\ge0}$, they consider
\[
 M_p(D_\lambda):=\frac{R}{(f)+I_p(D_\lambda)},
 \qquad
 m_p(D_\lambda):=\dim_{\C}M_p(D_\lambda).
\]
Our quotient is obtained by imposing the Jacobian ideal as well:
\[
 M_p(D_\lambda)=\frac{R}{(f)+I_p(D_\lambda)}
 \twoheadrightarrow
 \frac{R}{I_p(D_\lambda)+(f,J_f)}.
\]
Thus $H_f(p+\lambda)$ should be viewed as the Tjurina reduction of the Hodge moduli algebra at $(p,\lambda)$. The extra quotient by $J_f$ has the advantage that the images of all Hodge ideals $I_p(D_\lambda)$ are considered inside the same finite-dimensional Tjurina algebra $T_f=R/(f,J_f)$. Hence the family of quotients defining $H_f(p+\lambda)$ can be interpreted through a single filtration on $T_f$, leading to the strict spectral cutoff formula below.

\begin{maintheorem}\label{thm:introA}
For every $c\in\Q_{>0}$, written uniquely as $c=p+\lambda$ with $p\in\Z_{\ge0}$ and $\lambda\in\Q\cap(0,1]$, we have
\begin{equation}\label{eq:count-intro}
 H_f(c):=\dim_{\C}\frac{R}{I_p(\lambda Z_f)+(f,J_f)}
 =\#\{i:\alpha^{\Tj}_{f,i}<c\}.
\end{equation}
\end{maintheorem}

Thus rational Hodge--Tjurina numbers are exactly strict rational cutoff counts for the Tjurina spectrum.

We now turn to the Thom--Sebastiani problem for the Hodge--Tjurina numbers introduced above. The definition of $H_f(c)$ is closely related to that of the Tjurina number $\tau_f=\dim_{\C}\frac{R}{(f,J_f)},$ in the sense that $H_f(c)$ measures the Hodge-theoretic refinement of the Tjurina algebra at a given rational level. The classical Thom--Sebastiani formula gives $\mu_{f+g}=\mu_f\mu_g$ where $f$ and $g$ are isolated hypersurface singularities in disjoint sets of variables. For the Tjurina number, however, the analogous product formula $\tau_{f+g}=\tau_f\tau_g$ generally fails. Almir\'on \cite{Almiron} studied the discrepancy $\tau_{f+g}-\tau_f\tau_g$. In the present paper, we describe this defect through a filtered space $K_{f,g}$; see \cref{sec:filteredST} for details. Its dimension recovers the numerical discrepancy, while the induced filtration records the contribution at each spectral level.

More generally, Thom--Sebastiani phenomena have been extensively studied for Hodge-theoretic invariants. The classical theorem of Sebastiani--Thom and Scherk--Steenbrink describes the spectrum of a separated sum \cite{ST,ScherkSteenbrink}, while Maxim--Saito--Sch\"urmann established a microlocal filtered version \cite{MSS}. Related tensor product phenomena also appear for Bernstein--Sato polynomials; see Shi--Zuo \cite{ShiZuo}.

Motivated by these results, we study the Thom--Sebastiani behavior of the rational Hodge--Tjurina numbers $H_f(c)$. From \cref{sec:ST} onward, let $\bx=(x_1,\ldots,x_n)$ and $\by=(y_1,\ldots,y_m)$ with $n,m\ge2$, let $f\in R_f=\C\{\bx\}$ and $g\in R_g=\C\{\by\}$ define isolated hypersurface singularities in disjoint variables, and set
\[
 F(\bx,\by)=f(\bx)+g(\by)\in\C\{\bx,\by\}.
\]
For $h=f,g,F$, we use the analogous notation $J_h,A_h,T_h,\mu_h,\tau_h,H_h(c)$, and $m_k(h)$. Separation of variables gives $A_F\simeq A_f\otimes_{\C}A_g$, while the classical Thom--Sebastiani theorem gives $\Sp_F(t)=\Sp_f(t)\Sp_g(t)$. We will review this part in \cref{sec:ST}.

At the Tjurina level, however, the class of the equation has to be killed. In $A=A_f\otimes_{\C}A_g$, put $u=[f]\otimes1$ and $v=1\otimes[g]$, where $[f]$ and $[g]$ denote the classes of $f$ and $g$ in the Milnor algebras. Then $T_F=A/(u+v)A$ and $T_f\otimes T_g=A/(uA+vA)$, so the failure of the naive Tjurina tensor product is measured by the canonical kernel
\begin{equation}\label{eq:Kintro}
 K_{f,g}:=\frac{uA+vA}{(u+v)A}=\ker\bigl(T_F\twoheadrightarrow T_f\otimes T_g\bigr).
\end{equation}
Almir\'on's work describes the resulting discrepancy at the level of the ordinary Tjurina number. We lift this numerical defect to the filtered and spectral levels by realizing it as the canonical kernel $K_{f,g}$ and equipping this kernel with the filtration induced from the microlocal quotient filtration on $T_F$. The total dimension of $K_{f,g}$ is then described in terms of the Jordan-block data of multiplication by the classes $[f]$ and $[g]$, while the induced filtration records how this defect is distributed among the spectral levels. See \eqref{eq:Kfiltration} for details. Here we set
\[
 \defect_{f,g}(t)=\sum_{\gamma}\dim_{\C}\Gr_V^{\gamma}K_{f,g}\,t^{\gamma},
 \qquad
 \eps_{f,g}=\dim_{\C}K_{f,g}=\defect_{f,g}(1).
\]
By a rational-exponent polynomial, we mean a finite sum $P(t)=\sum_{\gamma\in\Q}p_\gamma t^\gamma$ with $p_\gamma\in\mathbb{R}$. For two such polynomials, $P\preceq Q$ means coefficientwise inequality. Our convention for weighted homogeneity is given in \cref{sec:prelim}.

For an isolated hypersurface singularity $h$, let $N_h:A_h\to A_h$ be the multiplication map induced by the class $[h]$. We write
\[
 b_j(h):=\dim_{\C}\bigl(N_h^{j-1}A_h/N_h^jA_h\bigr),\qquad j\ge1,
\]
so that $b_j(h)$ is the number of Jordan blocks of $N_h$ of length at least $j$. This Jordan-block description is developed in \cref{sec:ST}.

\begin{maintheorem}\label{thm:introB}
For arbitrary $f,g,F$ as above,
\begin{equation}\label{eq:introA}
 \Sp_F^{\Tj}(t)=\Sp_f^{\Tj}(t)\Sp_g^{\Tj}(t)+\defect_{f,g}(t),
 \qquad \defect_{f,g}(t)\succeq0.
\end{equation}
In particular,
\[
 \Sp_f^{\Tj}(t)\Sp_g^{\Tj}(t)\preceq\Sp_F^{\Tj}(t)\preceq\Sp_f(t)\Sp_g(t).
\]
Taking dimensions in the filtered defect formula gives
\begin{equation}\label{eq:intro-eps}
 \eps_{f,g}=\tau_F-\tau_f\tau_g=\sum_{j\ge2}b_j(f)b_j(g).
\end{equation}
The exact product $\Sp_F^{\Tj}(t)=\Sp_f^{\Tj}(t)\Sp_g^{\Tj}(t)$ holds if and only if $f\in J_f$ or $g\in J_g$. If $a_h=\min_i\alpha^{\St}_{h,i}$ for $h=f,g$, then every exponent $\gamma$ occurring with nonzero coefficient in $\defect_{f,g}(t)$ satisfies $\gamma\ge a_f+a_g+1$.
\end{maintheorem}

The numerical identity \eqref{eq:intro-eps} gives a Jordan-block description of the ordinary Tjurina defect and recovers Almir\'on's numerical Thom--Sebastiani analysis of $\tau$. The polynomial $\defect_{f,g}(t)$ further refines this formula by recording the filtration level of each contribution. In the notation of \cite{Almiron}, $e_{\mathrm{BS}}(h):=\min\{q\ge1:h^q\in J_h\}$. For two plane-curve factors, the numerical square-zero formula in \cref{cor:squarezero} agrees with Almir\'on's $e_{\mathrm{BS}}=2$ specialization. When one factor is weighted homogeneous, the correction term vanishes and the formula specializes to $\Sp_F^{\Tj}(t)=\Sp_f(t)\Sp_g^{\Tj}(t)$ when $f$ is the weighted-homogeneous factor.

The following corollary answers the question raised in Almir\'on's Remark~3 \cite{Almiron}.
\begin{corollary}\label{cor:intro-ST-Tj-spectrum}
The Tjurina subspectrum satisfies the Thom--Sebastiani property, namely
\[
 \Sp_F^{\Tj}(t)=\Sp_f^{\Tj}(t)\Sp_g^{\Tj}(t),
\]
if and only if
\[
 \tau_F=\tau_f\tau_g.
\]
\end{corollary}

The filtered defect formula also yields exact cutoff formulas and bounds for the rational Hodge--Tjurina numbers. Here let $V^cK_{f,g}$ denote the induced microlocal $V$-filtration step on the defect space and define $E_{f,g}(c)=\dim_{\C}\bigl(K_{f,g}/V^cK_{f,g}\bigr)$; see \cref{sec:ST,sec:filteredST} for details. We also write $m_{-1}(h)=0$ and $\Delta_r(h)=m_r(h)-m_{r-1}(h)$ for $h=f,g,F$.

\begin{maintheorem}\label{thm:introC}
With notions above, let
\[
 C_{f,g}(c)=\#\{(i,j):\alpha^{\Tj}_{f,i}+\alpha^{\Tj}_{g,j}<c\}.
\]
Then for every $c\in\Q_{>0}$, we have
\begin{equation}\label{eq:intro-cutoff}
 H_F(c)=C_{f,g}(c)+E_{f,g}(c).
\end{equation}
The cutoff error satisfies
\[
 0\le E_{f,g}(c)\le\eps_{f,g},\qquad
 E_{f,g}(c)=0\quad\text{if }c\le a_f+a_g+1.
\]
Moreover, for $m_k(F)$, we have
\begin{equation}\label{eq:intro-integer-bounds}
 \sum_{\substack{r,s\ge0\\ r+s\le k-1}}\Delta_r(f)\Delta_s(g)
 \le m_k(F)\le
 \sum_{\substack{r,s\ge0\\ r+s\le k}}\Delta_r(f)\Delta_s(g)+\eps_{f,g}.
\end{equation}
\end{maintheorem}

The exact cutoff formula and the integer-bin bounds are proved in \cref{cor:cutoff,thm:integer-general}. There are two different sources of error. One comes from using only integer cutoffs, which loses the exact fractional positions of the spectral numbers. The other comes from the Thom--Sebastiani defect measured by $K_{f,g}$. \Cref{prop:rational-recovery} shows that a known common denominator recovers the full Tjurina spectrum from finitely many rational Hodge--Tjurina numbers.

\crefrange{sec:prelim}{sec:counting} establish the algebraic, filtered, and spectral preliminaries and prove the counting interpretation of $H_f(c)$. \crefrange{sec:ST}{sec:filteredST} develop the numerical and filtered Thom--Sebastiani formulas. \Cref{sec:bounds} derives integer and rational bounds, finite rational recovery, and the symmetry estimate for weighted-homogeneous self sums. \Cref{sec:nonrecover} contains further examples illustrating the information lost when only integer Hodge--Tjurina data are retained.

\section{Filtered preliminaries for Milnor and Tjurina algebras}\label{sec:prelim}

\subsection{Milnor and Tjurina algebras}
Retain the notation of the Introduction and write $R=\C\{\bx\}$. Let
$\mathfrak m=(x_1,\ldots,x_n)$. Since $V(J_f)=\{0\}$ as a germ,
$\sqrt{J_f}=\mathfrak m$, so $A_f$ is Artinian. Hence $A_f$ and its quotient $T_f$ are
finite-dimensional over $\C$.
If $[f]$ denotes the class of $f$ in $A_f$, then
\(T_f\simeq A_f/[f]A_f.\)
Indeed, the image of $(f,J_f)$ in $A_f$ is the principal ideal $[f]A_f$, and hence
\[
 0\longrightarrow[f]A_f\longrightarrow A_f\longrightarrow T_f\longrightarrow0
\]
is exact, and
\begin{equation}\label{eq:mu-tau}
 \mu_f-\tau_f=\dim_{\C}[f]A_f.
\end{equation}

We say that $f$ is weighted homogeneous of weighted degree $d_0$ with positive weights $w_1,\ldots,w_n$ if every monomial $\bx^{\nu}$ occurring in $f$, with $\nu=(\nu_1,\ldots,\nu_n)\in\Z_{\ge0}^n\subset\mathbb R^n$ and $\bx^{\nu}:=x_1^{\nu_1}\cdots x_n^{\nu_n}$, satisfies
\(\sum_{i=1}^n w_i\nu_i=d_0.\)

\subsection{Finite decreasing filtrations and strict quotients}
In the finite-dimensional linear-algebra arguments below, all filtrations are decreasing,
exhaustive, separated, and left-continuous, with finitely many rational jumping indices.
For such a filtered finite-dimensional vector space $(M,V)$,
\[
 V^{\alpha}M=M\quad(\alpha\ll0),\qquad
 V^{\alpha}M=0\quad(\alpha\gg0),\qquad
 V^{\alpha}M=\bigcap_{\beta<\alpha}V^{\beta}M.
\]
This finiteness convention applies to filtrations on finite-dimensional spaces such as the
Milnor and Tjurina algebras and their quotients; it is not imposed on the ambient
microlocal $V$-filtration on the local ring. For a filtered vector space $(M,V)$, put
\[
 V^{>\alpha}M:=\bigcup_{\beta>\alpha}V^{\beta}M,
 \qquad
 \Gr_V^{\alpha}M:=V^{\alpha}M/V^{>\alpha}M.
\]
Because the filtration is decreasing, the union defining $V^{>\alpha}M$ is again a vector
subspace.

If $N\subset M$, the quotient filtration on $M/N$ is defined by
\(V^{\alpha}(M/N):=\frac{V^{\alpha}M+N}{N}.\)
Equivalently, if $\pi:M\twoheadrightarrow Q$ is a surjection and $Q$ is endowed with the
quotient filtration, then
\(V^{\alpha}Q=\pi(V^{\alpha}M)\) for \(\alpha\in\Q\).
With this quotient filtration, $\pi$ is strict. More generally, a filtered map $u:M\to N$ is
called \emph{strict} if
\(u(V^{\alpha}M)=u(M)\cap V^{\alpha}N\)
for every $\alpha\in\Q$.

\begin{lemma}\label{lem:strict-exact}
Suppose
\[
 0\longrightarrow M'\xrightarrow{i}M\xrightarrow{\pi}M''\longrightarrow0
\]
is a short exact sequence of filtered finite-dimensional vector spaces and both maps are
strict. Then, for every $\alpha\in\Q$,
\[
 0\longrightarrow \Gr_V^{\alpha}M'\xrightarrow{\Gr i}\Gr_V^{\alpha}M
 \xrightarrow{\Gr\pi}\Gr_V^{\alpha}M''\longrightarrow0
\]
is exact.
\end{lemma}
\begin{proof}
Since $i$ and $\pi$ are strict, after identifying $M'$ with its image in $M$, we have
\[
 V^{\alpha}M'=M'\cap V^{\alpha}M,
 \qquad
 \pi(V^{\alpha}M)=V^{\alpha}M''
\]
for every $\alpha$. It follows that
\(V^{>\alpha}M'=M'\cap V^{>\alpha}M.\)
The strictness of $\pi$ immediately implies that
\(\Gr_V^{\alpha}M\longrightarrow \Gr_V^{\alpha}M''\)
is surjective. To prove exactness in the middle, let
\([x]\in\Gr_V^{\alpha}M\)
map to zero. Then $x\in V^{\alpha}M$ and
\(\pi(x)\in V^{>\alpha}M''.\)
Choose $\beta>\alpha$ such that $\pi(x)\in V^{\beta}M''$. By strictness of $\pi$, there
exists $y\in V^{\beta}M\subseteq V^{>\alpha}M$ such that
\(\pi(y)=\pi(x).\)
Hence
\(x-y\in\ker\pi=M'.\)
Since $x-y\in V^{\alpha}M$, strictness of $i$ gives
\(x-y\in M'\cap V^{\alpha}M=V^{\alpha}M'.\)
Since $y\in V^{>\alpha}M$,
\([x]=[x-y]\quad\text{in }\Gr_V^{\alpha}M.\)
Thus
\(\ker(\Gr\pi)=\im(\Gr i),\)
where the reverse inclusion follows from $\pi\circ i=0$.

Finally, if
\([x']\in\Gr_V^{\alpha}M'\)
maps to zero in $\Gr_V^{\alpha}M$, then
\(x'\in M'\cap V^{>\alpha}M=V^{>\alpha}M'.\)
Then we have $[x']=0$, so $\Gr i$ is injective. The associated-graded sequence is therefore exact.
\end{proof}

\subsection{Convolution filtrations on tensor products}
Let $(M,V)$ and $(N,V)$ be finite-dimensional filtered vector spaces. Endow
$M\otimes_{\C}N$ with the convolution filtration
\begin{equation}\label{eq:convolution}
 V^c(M\otimes N):=\sum_{a+b\ge c}V^aM\otimes V^bN.
\end{equation}
The following elementary lemma is the filtered-linear-algebra input used repeatedly in the
Thom--Sebastiani arguments below.

\begin{lemma}\label{lem:conv-gr}
For every $c\in\Q$, there is a natural isomorphism
\begin{equation}\label{eq:conv-gr}
 \Gr_V^c(M\otimes N)\simeq
 \bigoplus_{a+b=c}\Gr_V^aM\otimes\Gr_V^bN.
\end{equation}
Accordingly, if
\(P_{M,V}(t):=\sum_a\dim_{\C}\Gr_V^aM\,t^a,\)
then
\begin{equation}\label{eq:HP-product}
 P_{M\otimes N,V}(t)=P_{M,V}(t)P_{N,V}(t).
\end{equation}
\end{lemma}
\begin{proof}
For $a,b\in\Q$, there is a canonical map
\[
 \Gr_V^aM\otimes\Gr_V^bN
 \longrightarrow
 \Gr_V^{a+b}(M\otimes N),
 \qquad
 [m]\otimes[n]\longmapsto[m\otimes n].
\]
It is well defined: if, for example, $m$ is changed by an element of $V^{>a}M$, then
the corresponding change in $m\otimes n$ belongs to $V^{>a+b}(M\otimes N)$, and
similarly for $n$. For each $c$, these maps therefore assemble to a natural morphism
\[
 \bigoplus_{a+b=c}\Gr_V^aM\otimes\Gr_V^bN
 \longrightarrow
 \Gr_V^c(M\otimes N).
\]
We show that it is an isomorphism by choosing bases adapted to the filtrations.

Let
\(e_1,\ldots,e_r\) and \(u_1,\ldots,u_s\)
be bases of $M$ and $N$, respectively, adapted to the filtrations. Write
\(\level(e_i)=a_i\) and \(\level(u_j)=b_j\).
By definition,
\[
 e_i\in V^{a_i}M\setminus V^{>a_i}M,
 \qquad
 u_j\in V^{b_j}N\setminus V^{>b_j}N,
\]
and the adapted-basis property gives
\[
 V^{\alpha}M=\operatorname{Span}\{e_i:a_i\ge\alpha\},
 \qquad
 V^{\beta}N=\operatorname{Span}\{u_j:b_j\ge\beta\}.
\]
To compute the convolution filtration explicitly, consider the tensor basis $\{e_i\otimes u_j\}_{i,j}$.
For each pair $(a,b)$,
\[
 V^aM\otimes V^bN
 =\operatorname{Span}\{e_i\otimes u_j:a_i\ge a,\ b_j\ge b\}.
\]
Therefore
\[
 V^c(M\otimes N)
 =\operatorname{Span}\{e_i\otimes u_j:a_i+b_j\ge c\}.
\]
Indeed, if $a_i\ge a$, $b_j\ge b$, and $a+b\ge c$, then
$a_i+b_j\ge a+b\ge c$. Conversely, if $a_i+b_j\ge c$, taking $a=a_i$ and
$b=b_j$ gives
\[
 e_i\otimes u_j\in V^{a_i}M\otimes V^{b_j}N\subseteq V^c(M\otimes N).
\]
Hence
\(\level(e_i\otimes u_j)=a_i+b_j,\)
and the tensor basis is adapted to the convolution filtration.

Now we have
\(\Gr_V^c(M\otimes N)=V^c(M\otimes N)/V^{>c}(M\otimes N).\)
A basis tensor $e_i\otimes u_j$ contributes a nonzero class precisely when its level is
$c$, that is, when $a_i+b_j=c$. Thus the classes
\([e_i\otimes u_j]\), with \(a_i+b_j=c\),
form a basis of $\Gr_V^c(M\otimes N)$. On the other hand, the classes $[e_i]$ with
$a_i=a$ form a basis of $\Gr_V^aM$, and the classes $[u_j]$ with $b_j=b$ form a basis
of $\Gr_V^bN$. For $a+b=c$, the tensors
\([e_i]\otimes[u_j]\), with \(a_i=a\) and \(b_j=b\),
form a basis of $\Gr_V^aM\otimes\Gr_V^bN$. The canonical map above identifies these
two bases and is therefore an isomorphism, proving \eqref{eq:conv-gr}.

Taking dimensions gives
\[
 \dim_{\C}\Gr_V^c(M\otimes N)
 =\sum_{a+b=c}
 \dim_{\C}\Gr_V^aM\,\dim_{\C}\Gr_V^bN.
\]
The Cauchy product formula then yields
\[
 \begin{aligned}
 P_{M\otimes N,V}(t)
 &=\sum_c\dim_{\C}\Gr_V^c(M\otimes N)t^c\\
 &=\left(\sum_a\dim_{\C}\Gr_V^aM\,t^a\right)
   \left(\sum_b\dim_{\C}\Gr_V^bN\,t^b\right),
 \end{aligned}
\]
which is \eqref{eq:HP-product}.
\end{proof}

\begin{lemma}\label{lem:tensor-strict}
Let $\pi:N\twoheadrightarrow Q$ be a strict filtered quotient, and endow $M\otimes N$
and $M\otimes Q$ with the convolution filtrations. Then
$\mathrm{id}_M\otimes\pi:M\otimes N\twoheadrightarrow M\otimes Q$ is strict, and
\begin{equation}\label{eq:tensor-strict}
 (\mathrm{id}_M\otimes\pi)\bigl(V^c(M\otimes N)\bigr)=V^c(M\otimes Q)=\sum_{a+b\ge c}V^aM\otimes V^bQ.
\end{equation}
\end{lemma}
\begin{proof}
Since $\pi$ is strict and surjective,
\(\pi(V^bN)=V^bQ\)
for every $b$. The image of the $c$-th convolution level is therefore
\[
 \begin{aligned}
 (\mathrm{id}_M\otimes\pi)\bigl(V^c(M\otimes N)\bigr)
 &=\sum_{a+b\ge c}V^aM\otimes\pi(V^bN)\\
 &=\sum_{a+b\ge c}V^aM\otimes V^bQ.
 \end{aligned}
\]
The right-hand side is precisely the quotient/convolution filtration on $M\otimes Q$,
proving both \eqref{eq:tensor-strict} and strictness.
\end{proof}

\cref{lem:conv-gr,lem:tensor-strict} remove a potential strictness ambiguity used later: once a quotient is given its quotient filtration, tensoring with another finite filtered vector space produces exactly the convolution filtration above.

\section{The JKSY filtration and spectral conventions}\label{sec:JKSY}

The microlocal $V$-filtration used here goes back to Saito \cite{MSaito}. It is a decreasing
$\Q$-indexed filtration $\widetilde V_f^\alpha R$ on the local ring. We use the indexing
convention of JKSY; see \cite[\S1.2]{JKSY} and the references therein, and we do not
recall the construction here. Its quotient on the Milnor algebra is
\begin{equation}\label{eq:V-A}
V_f^\alpha A_f
:=
\frac{\widetilde V_f^\alpha R+J_f}{J_f}.
\end{equation}
The microlocal filtration $\widetilde V_f^\alpha$ is kept in the Musta\c{t}\u{a}--Popa/JKSY convention used in the comparison below; the finite filtrations are recorded with the left-continuous convention fixed in \cref{sec:prelim}.

To relate this filtration to the spectral $V$-filtration, let $\Omega_R^k$ denote the
$R$-module of holomorphic $k$-forms and set
$\Omega_f^n:=\Omega_R^n/(df\wedge\Omega_R^{n-1}).$
The chosen volume form
$d\mathbf{x}:=dx_1\wedge\cdots\wedge dx_n$
induces an isomorphism
\[
A_f \longrightarrow \Omega_f^n,
\qquad
[g]\longmapsto [g\,d\mathbf{x}].
\]
Under this identification, JKSY identify the quotient filtration above with the quotient
$V$-filtration on $\Omega_f^n$ coming from the Brieskorn lattice. Together with their
Brieskorn-lattice description of the spectrum, this gives the Steenbrink multiplicities;
see \cite[Introduction, (0.1), and Proposition 1.3]{JKSY}. With the positive-exponent
convention used throughout this paper, we have
\begin{equation}\label{eq:St-mult}
 \#\{i:\alpha^{\St}_{f,i}=\beta\}
 =\dim_{\C}\Gr_{V_f}^{\beta}A_f,
\end{equation}
and
\(\Sp_f(t)=\sum_{i=1}^{\mu_f}t^{\alpha^{\St}_{f,i}}.\)

Hodge ideals were introduced by Musta\c{t}\u{a} and Popa for reduced divisors
\cite{MP-HI} and for effective $\Q$-divisors in \cite{MP-QBir}. For the Tjurina
subspectrum, we follow JKSY directly. Let $\alpha\in\Q\cap(0,1]$, set
$D_\alpha:=\alpha Z$, and put $\beta:=1-\alpha$. In the local situation used here,
the original $\Q$-divisor definition is formulated using
\[
 M(f^{\beta}):=R[f^{-1}]f^{\beta},
\]
with
\[
 \partial_{x_i}(u f^{\beta})
 =\bigl(\partial_{x_i}u+\beta u f^{-1}\partial_{x_i}f\bigr)f^{\beta}.
\]
Since $f$ is reduced and $Z=\operatorname{div}(f)$, the $q$-th Hodge ideal is defined by
\begin{equation}\label{eq:Hodge-ideal-original}
 F_qM(f^{\beta})
 =I_q(D_\alpha)\otimes_R R((q+1)Z)f^{\beta}
 =I_q(D_\alpha)f^{-q-1}f^{\beta};
\end{equation}
see \cite[definition after Proposition~4.1]{MP-QBir}. For $\alpha=1$, this becomes
\(F_qR[f^{-1}]=I_q(Z)f^{-q-1}.\)
The later comparison with the microlocal $V$-filtration is due to Musta\c{t}\u{a}--Popa
\cite{MP-Q}. Following \cite[Introduction, (0.2)]{JKSY}, we use the non-strict condition $\alpha+q\ge\beta$ for the Hodge-ideal filtration. This is the left-continuous convention compatible with $\Gr_V^\beta=V^\beta/V^{>\beta}$ and with the strict cutoff in the counting formulas below. JKSY define the Hodge-ideal filtration on
$A_f$ by
\[
 V_{\HI}^{\beta}A_f
 :=\sum_{\alpha+q\ge\beta}I_q(\alpha Z)\bmod J_f,
\]
and use its quotient on $T_f$:
\begin{equation}\label{eq:VHI-T}
 V_{\HI}^{\beta}T_f
 =\frac{\sum_{\alpha+q\ge\beta}I_q(\alpha Z)+(f,J_f)}{(f,J_f)}.
\end{equation}
The Tjurina subspectrum is then defined by JKSY from the quotient filtration on
$R/(f,J_f)$; see \cite[Introduction, (0.2)--(0.4)]{JKSY}. The Tjurina spectral exponents are therefore characterized by
\begin{equation}\label{eq:Tj-mult}
 \#\{i:\alpha^{\Tj}_{f,i}=\beta\}
 =\dim_{\C}\Gr_{V_{\HI}}^{\beta}T_f,
\end{equation}
and
\(\Sp_f^{\Tj}(t)=\sum_{i=1}^{\tau_f}t^{\alpha^{\Tj}_{f,i}}.\)

\begin{remark}\label{rem:JKSY-convention}
For a weighted-homogeneous isolated singularity of weighted degree $1$, with positive
weights $w_1,\ldots,w_n$, JKSY use the weighted level
\[
 \ell_w(\nu):=\sum_{j=1}^n w_j(\nu_j+1),
 \qquad
 \bx^{\nu}=\prod_jx_j^{\nu_j}.
\]
We use the corresponding positive spectral convention. \cref{rem:rho} recalls the
Brieskorn-lattice calculation and the precise identification with the spectral numbers.
\end{remark}

\begin{example}\label{ex:cusp}
Let $f=x^2+y^3$, with normalized weights $w_x=1/2$ and $w_y=1/3$. Then
\(A_f=\C\{x,y\}/(x,y^2)\simeq\C[y]/(y^2)\)
has basis $1,y$. Under the Brieskorn-lattice description of the quotient $V$-filtration,
the corresponding classes $dx\wedge dy$ and $y\,dx\wedge dy$ have levels
\(\frac12+\frac13=\frac56\) and \(\frac12+\frac23=\frac76\).
Hence
\(\Sp_f(t)=t^{5/6}+t^{7/6}.\)

\end{example}

If $f$ is weighted homogeneous, Euler's identity gives $f\in J_f$, hence $T_f=A_f$.
The two quotient filtrations therefore agree, so the Tjurina and Steenbrink spectra coincide with multiplicities; compare \cite[Introduction, (0.5)]{JKSY}.

\section{The Hodge--Tjurina spectral counting theorem}\label{sec:counting}

The only external comparison needed for the counting argument comes from the
Hodge-ideal theory of Musta\c{t}\u{a}--Popa. Let
\[
D_\lambda:=\lambda Z=\lambda\,\operatorname{div}(f),
\qquad
\lambda\in\Q\cap(0,1].
\]
Following Musta\c{t}\u{a}--Popa \cite{MP-Q}, the modified Hodge ideal is
defined by
\[
I_p''(D_\lambda)
:=
I_p(D_\lambda)\,
\mathcal O_X\!\bigl(p(Z-\operatorname{div}(f))\bigr).
\]
Musta\c{t}\u{a}--Popa \cite[Theorem A]{MP-Q} prove the modulo-$(f)$ comparison for $p\ge1$. For $p=0$, we use the corresponding multiplier-ideal identity, also covered by \cite[Corollary~2.5]{JKSY}.
\begin{equation}\label{eq:MP-comparison}
I_p''(D_\lambda)+(f)
=
\widetilde V_f^{p+\lambda}R+(f),
\qquad
p\in\Z_{\ge0},\quad
\lambda\in\Q\cap(0,1].
\end{equation}

In the present paper, $f$ is reduced and $Z=\operatorname{div}(f)$.
Hence
\[
Z-\operatorname{div}(f)=0,
\]
so the correcting twist is trivial and therefore
\[
I_p''(D_\lambda)
=
I_p(D_\lambda)
=
I_p(\lambda Z).
\]
Consequently, \eqref{eq:MP-comparison} becomes
\[
I_p(\lambda Z)+(f)
=
\widetilde V_f^{p+\lambda}R+(f).
\]
Since $\lambda\in\Q\cap(0,1]$, we also have
$\lceil D_\lambda\rceil=Z$.
The comparison is only modulo $(f)$, which is exactly what is needed after
passing to the Tjurina quotient, where $f$ vanishes.

\begin{theorem}\label{thm:counting}
Let $f\in R$ be reduced and define an isolated hypersurface singularity $Z=(f=0)$.
For $p\in\Z_{\ge0}$ and $\lambda\in\Q\cap(0,1]$,
\begin{equation}\label{eq:VHI-cutoff}
 \frac{I_p(\lambda Z)+(f,J_f)}{(f,J_f)}=V_{\HI}^{p+\lambda}T_f.
\end{equation}
It follows that
\begin{equation}\label{eq:counting}
 \dim_{\C}\frac{R}{I_p(\lambda Z)+(f,J_f)}
 =\#\{i:\alpha^{\Tj}_{f,i}<p+\lambda\}.
\end{equation}
\end{theorem}
\begin{proof}
By \eqref{eq:VHI-T},
\[
 V_{\HI}^{p+\lambda}T_f
 =\frac{\sum_{\alpha+q\ge p+\lambda}I_q(\alpha Z)+(f,J_f)}{(f,J_f)}.
\]
For every summand, \eqref{eq:MP-comparison} and the fact that the microlocal
$V$-filtration is decreasing give
\[
 I_q(\alpha Z)+(f)
 =\widetilde V_f^{q+\alpha}R+(f)
 \subseteq\widetilde V_f^{p+\lambda}R+(f)
 =I_p(\lambda Z)+(f).
\]
After adjoining $J_f$, the whole sum is contained in $I_p(\lambda Z)+(f,J_f)$. The
reverse inclusion holds because the pair $(q,\alpha)=(p,\lambda)$ occurs in the sum.
This proves \eqref{eq:VHI-cutoff}.

The third isomorphism theorem gives
\[
 \frac{R}{I_p(\lambda Z)+(f,J_f)}
 \simeq
 \frac{T_f}{V_{\HI}^{p+\lambda}T_f}.
\]
Using \eqref{eq:Tj-mult} and the fact that the filtration is decreasing,
\[
 \dim_{\C}\frac{T_f}{V_{\HI}^{p+\lambda}T_f}
 =\sum_{\beta<p+\lambda}\dim_{\C}\Gr_{V_{\HI}}^{\beta}T_f
 =\#\{i:\alpha^{\Tj}_{f,i}<p+\lambda\}.
\]
So we have $\dim_{\C}\frac{R}{I_p(\lambda Z)+(f,J_f)}=\#\{i:\alpha^{\Tj}_{f,i}<p+\lambda\}.$

\end{proof}

\begin{corollary}\label{cor:mk-count}
For
\(m_k(f):=\dim_{\C}\frac{R}{I_k(Z)+(f,J_f)},\)
one has, for every $k\in\Z_{\ge0}$,
\begin{equation}\label{eq:mk-count}
 m_k(f)=\#\{i:\alpha^{\Tj}_{f,i}<k+1\}.
\end{equation}
\end{corollary}
\begin{proof}
This is \cref{thm:counting} with $p=k$ and $\lambda=1$.
\end{proof}

\begin{remark}\label{rem:rho}
Suppose that $f$ is weighted homogeneous of weighted degree $1$ with
positive weights $w_1,\ldots,w_n$ and defines an isolated
hypersurface singularity. Since $A_f=R/J_f$ is finite-dimensional,
we may choose a monomial basis
$[v_1],\ldots,[v_{\mu_f}]$ of $A_f$ compatible with the weighted grading. Write
$v_j=\bx^{\nu_j}=x_1^{\nu_{j,1}}\cdots x_n^{\nu_{j,n}}$,
and set $ \rho(v_j):=\sum_{\ell=1}^n w_\ell(\nu_{j,\ell}+1).$

In JKSY, \S1.5, the quantity corresponding to
$\rho(v_j)$ is denoted by
$ \ell_w(\nu):=\sum_{i=1}^n w_i(\nu_i+1).$
Let $d\bx:=dx_1\wedge\cdots\wedge dx_n$. JKSY prove in \cite[(1.21)]{JKSY} that
\begin{equation}
 \partial_t t[\bx^{\nu}d\bx]
 =\ell_w(\nu)[\bx^{\nu}d\bx],
\end{equation}
where $t$ is the base parameter in the Gauss--Manin system. Assigning weighted degree
$w_i$ to both $x_i$ and $dx_i$, we have
\[
 \deg_w(\bx^{\nu}d\bx)
 =\sum_i w_i\nu_i+\sum_i w_i
 =\ell_w(\nu).
\]
Thus the $+1$ in each term $w_i(\nu_i+1)$ comes from the contribution of the volume
form $d\bx$, rather than from an additional indexing shift.

The $V$-filtration on the Gauss--Manin system is indexed so that
$\partial_t t-\beta$ is nilpotent on $\Gr_V^\beta$; see \cite[\S1.2]{JKSY}.
With this convention, the identity above places the class $[\bx^{\nu}d\bx]$ at
$V$-level $\ell_w(\nu)$. JKSY make this explicit by introducing the weighted filtration
$V_w$ on $\C\{\bx\}$ in \cite[(1.22)]{JKSY} and identifying the $V$-filtration on the
Brieskorn lattice with the quotient filtration induced by $V_w$ in \cite[(1.23)]{JKSY}.
Combining this identification with their Hilbert--Poincar\'e description
\cite[(1.24)]{JKSY} and the weighted-homogeneous spectrum formula
\cite[(1.20)]{JKSY} gives the equality of multisets
\[
 \{\alpha^{\St}_{f,j}\}_{j=1}^{\mu_f}
 =\{\rho(v_j)\}_{j=1}^{\mu_f}.
\]
After relabeling the monomial basis if necessary, we may therefore assume that
\[
\alpha^{\St}_{f,j}=\rho(v_j),
\qquad
1\le j\le\mu_f.
\]

On the other hand, Euler's identity gives $f\in J_f$, and hence $T_f=A_f$. Under this
identification, the quotient filtration on $T_f$ agrees with the filtration on $A_f$.
Therefore the Tjurina subspectrum coincides with the Steenbrink spectrum; compare
\cite[Introduction, (0.5)]{JKSY}. Consequently,
\[
 \{\alpha^{\Tj}_{f,j}\}_{j=1}^{\mu_f}
 =\{\alpha^{\St}_{f,j}\}_{j=1}^{\mu_f}
 =\{\rho(v_j)\}_{j=1}^{\mu_f}.
\]

\cref{cor:mk-count} therefore specializes to
\[
 m_k(f)=\#\{j:\rho(v_j)<k+1\}.
\]
Thus, for a weighted-homogeneous isolated hypersurface singularity, $m_k(f)$ counts the
Milnor-algebra basis monomials whose weighted spectral level is strictly below $k+1$.
\end{remark}
We conclude the section with a simple weighted-homogeneous example illustrating the counting formula. Let
\(f=x^a+y^b\) with \(a,b\ge2\).
Then $f$ is weighted homogeneous of degree $1$ with weights $(w_1,w_2)=(\frac{1}{a},\frac{1}{b})$, and
\(A_f=T_f\simeq\C\{x,y\}/(x^{a-1},y^{b-1}).\)
The monomial classes $x^iy^j$, $0\le i\le a-2$, $0\le j\le b-2$, have spectral
exponents
\(\alpha_{i,j}=\frac{i+1}{a}+\frac{j+1}{b}.\)
\cref{cor:mk-count} gives
\begin{equation}\label{eq:BP-mk}
 m_k(f)=\#\left\{(i,j):0\le i\le a-2,\ 0\le j\le b-2,
 \ \frac{i+1}{a}+\frac{j+1}{b}<k+1\right\}.
\end{equation}
For $k\ge1$, every exponent is less than 2, since
\(\frac{a-1}{a}+\frac{b-1}{b}<2,\)
so
\(m_k(f)=\mu_f=(a-1)(b-1).\)
For $k=0$, however, \eqref{eq:BP-mk} retains the nontrivial part of the spectrum below
$1$. For instance, for $f=x^2+y^5$ the exponents are
\(\frac7{10},\ \frac9{10},\ \frac{11}{10},\ \frac{13}{10},\)
and hence
\(m_0(f)=2\) and \(m_k(f)=4\) for \(k\ge1\).

The endpoint can be tested explicitly on the ordinary node
\(h=x^2+y^2.\)
Here $A_h=T_h\simeq\C$, and the unique spectral exponent in our convention is $1$.
\cref{cor:mk-count} therefore gives
\(m_0(h)=0\) and \(m_k(h)=1\) for \(k\ge1\).
The strict inequality in the counting theorem is therefore essential: replacing $\alpha<1$ by
$\alpha\le1$ would incorrectly give $m_0(h)=1$.

\section{Sebastiani--Thom sums}\label{sec:ST}

Let
\(\bx=(x_1,\ldots,x_n)\) and \(\by=(y_1,\ldots,y_m)\) with $n,m\ge2$,
and let
\(f\in R_f=\C\{\bx\}\) and \(g\in R_g=\C\{\by\}\)
define isolated hypersurface singularities in disjoint sets of variables. Set
\(F(\bx,\by)=f(\bx)+g(\by)\) and \(R_F=\C\{\bx,\by\}\).
For $h=f,g,F$, write
$J_h$ for the Jacobian ideal and set
\[
 A_h=R_h/J_h,\qquad T_h=R_h/(h,J_h),\qquad
 \mu_h=\dim_{\C}A_h,\qquad \tau_h=\dim_{\C}T_h.
\]
We also write $Z_h=(h=0)$. Since the variables are disjoint,
\(J_F=(J_f,J_g),\)
so $F$ has an isolated critical point. This is the classical Sebastiani--Thom construction;
see \cite{ST}.

At the level of convergent power-series rings, one should not replace $R_F$ by the
algebraic tensor product $R_f\otimes_{\C}R_g$ without completion. We only use tensor
products after quotienting by the Jacobian ideals. Since $A_f$ and $A_g$ are
finite-dimensional, the natural multiplication map induces a canonical isomorphism
\((R_f/J_f)\otimes_{\C}(R_g/J_g) \longrightarrow R_F/(J_f,J_g).\)
Indeed, representatives of bases of the two Milnor algebras give a product basis modulo
$(J_f,J_g)$. Passing to the Jacobian quotients gives the canonical identification
\begin{equation}\label{eq:AF-tensor}
 A_F\simeq A_f\otimes_{\C}A_g,
\end{equation}
under which
\begin{equation}\label{eq:Fclass}
 [F]=[f]\otimes1+1\otimes[g].
\end{equation}
Consequently,
\begin{equation}\label{eq:TF-coker}
 T_F\simeq
 \frac{A_f\otimes A_g}{([f]\otimes1+1\otimes[g])(A_f\otimes A_g)}.
\end{equation}

We next describe the numerical Tjurina defect in terms of the Jordan blocks of multiplication by $[f]$ and $[g]$. For $h=f,g$, let
\(N_h:A_h\longrightarrow A_h\)
be multiplication by $[h]$. This operator is nilpotent: $A_h$ is local Artinian and $[h]$
belongs to its nilpotent maximal ideal. Its Jordan blocks therefore have positive sizes
\(\ell_{h,1},\ldots,\ell_{h,\tau_h}.\)
To make the relation with the Tjurina number explicit, consider first a single nilpotent
Jordan block of length $\ell$. Choose a basis $e_1,\ldots,e_\ell$ such that
\(N e_1=0\), \(N e_2=e_1\), \(\ldots\), and \(N e_\ell=e_{\ell-1}\).
Then
\(NA=\langle e_1,\ldots,e_{\ell-1}\rangle,\)
so the cokernel of this block is one-dimensional. Since $T_h=A_h/[h]A_h=\coker N_h$,
each Jordan block contributes exactly one dimension to $T_h$. Hence the number of blocks is
\(\tau_h=\dim_{\C}\coker N_h.\)

For the same block,
\(N^{j-1}A/N^jA\)
is one-dimensional for $j\le\ell$ and zero for $j>\ell$. Therefore, after taking the direct
sum over all blocks, we have
\begin{equation}\label{eq:bj}
 b_j(h):=\dim_{\C}\bigl(N_h^{j-1}A_h/N_h^jA_h\bigr)
 =\#\{i:\ell_{h,i}\ge j\},\qquad j\ge1.
\end{equation}
Thus a block of length $\ell$ contributes $1$ to each of
$b_1(h),\ldots,b_\ell(h)$ and nothing to the later $b_j(h)$. Consequently,
\[
 b_1(h)=\tau_h,
 \qquad
 \sum_{j\ge1}b_j(h)=\sum_i\ell_{h,i}=\dim_{\C}A_h=\mu_h.
\]
Equivalently,
\begin{equation}\label{eq:mu-tau-blocks}
 \mu_h-\tau_h
 =\sum_i(\ell_{h,i}-1)
 =\sum_{j\ge2}b_j(h).
\end{equation}
In this sense, $\mu_h-\tau_h$ measures the total amount by which the Jordan blocks exceed
length one.

\begin{proposition}\label{prop:JordanDefect}
With the notation above for the Jordan blocks of $N_f$ and $N_g$
and with $F=f+g$, we have
\begin{equation}\label{eq:tau-Jordan}
 \tau_F=
 \sum_{i=1}^{\tau_f}\sum_{a=1}^{\tau_g}\min(\ell_{f,i},\ell_{g,a})
 =\sum_{j\ge1}b_j(f)b_j(g),
\end{equation}
and
\begin{equation}\label{eq:eps-Jordan}
 \eps_{f,g}:=\tau_F-\tau_f\tau_g
 =\sum_{j\ge2}b_j(f)b_j(g).
\end{equation}
Therefore
\begin{equation}\label{eq:eps-upper}
 0\le\eps_{f,g}\le(\mu_f-\tau_f)(\mu_g-\tau_g).
\end{equation}
The defect vanishes exactly when $N_f=0$ or $N_g=0$.
\end{proposition}
\begin{proof}
Write the Jordan decompositions as
\[
 A_f=\bigoplus_{i=1}^{\tau_f}V_{f,i},
 \qquad
 A_g=\bigoplus_{a=1}^{\tau_g}W_{g,a},
\]
where $V_{f,i}$ and $W_{g,a}$ are blocks of lengths $\ell_{f,i}$ and $\ell_{g,a}$
respectively. Under the identification $A_F\simeq A_f\otimes A_g$, multiplication by
$[F]=[f]\otimes1+1\otimes[g]$ is
\(N_F=N_f\otimes1+1\otimes N_g.\)
Hence $A_F$ is the direct sum of the tensor products $V_{f,i}\otimes W_{g,a}$, and it is
enough to compute the cokernel of $N_F$ on one such pair of blocks. A block of size $a$ for $N_f$ may be written as the vector space
$\C[s]/(s^a)$, on which $N_f$ acts by multiplication by $s$; equivalently, it has a
basis $1,s,\ldots,s^{a-1}$. Likewise a block of size $b$ for $N_g$ is
$\C[t]/(t^b)$. On a pure tensor $p(s)\otimes q(t)$, we have
\[
 (N_f\otimes1)(p\otimes q)=sp\otimes q,
 \qquad
 (1\otimes N_g)(p\otimes q)=p\otimes tq.
\]
Thus, after identifying the tensor product with
\(\C[s,t]/(s^a,t^b),\)
the operator $N_f\otimes1+1\otimes N_g$ is multiplication by $s+t$. Its cokernel is
\[
 \frac{\C[s,t]}{(s^a,t^b,s+t)}
 \simeq
 \frac{\C[s]}{(s^{\min(a,b)})},
\]
and therefore has dimension $\min(a,b)$. Summing over all pairs of blocks gives
\[
 \dim_{\C}\coker N_F
 =\sum_{i=1}^{\tau_f}\sum_{a=1}^{\tau_g}
   \min(\ell_{f,i},\ell_{g,a}).
\]
Since $T_F=\coker N_F$, this is the first equality in \eqref{eq:tau-Jordan}.

For positive integers $a,b$,
\(\min(a,b) =\sum_{j\ge1}\ind_{\{a\ge j\}}\ind_{\{b\ge j\}}.\)
Applying this to every pair of blocks and rearranging the finite sums gives
\begin{align*}
 \tau_F
 &=\sum_{i,a}\sum_{j\ge1}
   \ind_{\{\ell_{f,i}\ge j\}}\ind_{\{\ell_{g,a}\ge j\}}\\
 &=\sum_{j\ge1}
   \left(\sum_i\ind_{\{\ell_{f,i}\ge j\}}\right)
   \left(\sum_a\ind_{\{\ell_{g,a}\ge j\}}\right)\\
 &=\sum_{j\ge1}b_j(f)b_j(g),
\end{align*}
which is the second equality in \eqref{eq:tau-Jordan}. Because
$b_1(f)=\tau_f$ and $b_1(g)=\tau_g$, the $j=1$ term is $\tau_f\tau_g$, and hence
\(\eps_{f,g}=\tau_F-\tau_f\tau_g =\sum_{j\ge2}b_j(f)b_j(g).\)

For the upper bound, \eqref{eq:mu-tau-blocks} gives
\begin{align*}
 (\mu_f-\tau_f)(\mu_g-\tau_g)
 &=\left(\sum_{j\ge2}b_j(f)\right)
   \left(\sum_{k\ge2}b_k(g)\right)\\
 &=\sum_{j,k\ge2}b_j(f)b_k(g)\\
 &\ge\sum_{j\ge2}b_j(f)b_j(g)
 =\eps_{f,g},
\end{align*}
which proves \eqref{eq:eps-upper}. Finally, if both $N_f$ and $N_g$ are nonzero, each
has a block of length at least two, so $b_2(f)b_2(g)>0$ and the defect is positive.
Conversely, if one operator is zero, every block for that operator has length one and all its
$b_j$ with $j\ge2$ vanish. This proves the vanishing criterion.
\end{proof}

The numerical problem in \eqref{eq:tau-Jordan} belongs to the Thom--Sebastiani theory
of Tjurina numbers developed in \cite{Almiron}. The block argument above is included to
make the defect calculation self-contained. No compatibility between Jordan bases and the
$V$-filtration is assumed.

\begin{corollary}\label{cor:squarezero}
One has $\tau_F=\tau_f\tau_g$ if and only if $f\in J_f$ or $g\in J_g$. If $N_f^2=N_g^2=0$, then
\begin{equation}\label{eq:squarezero}
 \tau_F=\tau_f\tau_g+(\mu_f-\tau_f)(\mu_g-\tau_g).
\end{equation}
For two reduced isolated plane curve singularities, the same formula \eqref{eq:squarezero} applies.
\end{corollary}
\begin{proof}
The first assertion follows from \cref{prop:JordanDefect} and K.~Saito's
criterion \cite{KSaito}: $N_h=0$ is equivalent to $h\in J_h$, equivalently
$\mu_h=\tau_h$. If $N_h^2=0$, then $b_j(h)=0$ for $j\ge3$, so
\(b_2(h)=\mu_h-\tau_h,\)
and \eqref{eq:squarezero} follows from \eqref{eq:eps-Jordan}. For a germ in $d$
variables, the classical Brian\c{c}on--Skoda containment
\(h^d\in J_h\)
\cite{BS} gives $N_h^d=0$. Taking $d=2$ proves the plane-curve assertion.
\end{proof}

\section{Filtered Thom--Sebastiani and the Tjurina spectrum}\label{sec:filteredST}

The microlocal Thom--Sebastiani theorem of Maxim--Saito--Sch\"urmann gives, after
shrinking the representatives if necessary,
\begin{equation}\label{eq:microTS}
 \widetilde V_F^cR_F
 =\sum_{a+b\ge c}(\widetilde V_f^aR_f)R_F\,(\widetilde V_g^bR_g)R_F;
\end{equation}
see \cite[proof of Theorem 3.2, (3.2.3)]{MSS}. We use the inequality form because the
filtration is decreasing. Formula \eqref{eq:microTS} is a statement on the structure sheaf:
each factor $\widetilde V_f^aR_f$ or $\widetilde V_g^bR_g$ is first extended to $R_F$, and
the resulting ideals are multiplied in $R_F$. This formulation avoids any completion issue
at the analytic local-ring level. The tensor-product interpretation appears only after
applying the finite-dimensional Jacobian quotient.

\begin{proposition}\label{prop:MilnorConv}
Under \eqref{eq:AF-tensor},
\begin{equation}\label{eq:MilnorConv}
 V_F^cA_F=\sum_{a+b\ge c}V_f^aA_f\otimes V_g^bA_g.
\end{equation}
Taking associated graded gives
\begin{equation}\label{eq:MilnorConv-gr}
 \Gr_{V_F}^cA_F\simeq
 \bigoplus_{a+b=c}\Gr_{V_f}^aA_f\otimes\Gr_{V_g}^bA_g.
\end{equation}
\end{proposition}
\begin{proof}
Let
\(q_F:R_F\twoheadrightarrow A_F=R_F/J_F\)
be the natural quotient map. By definition, the filtration on the Milnor algebra $A_F$ is
the quotient filtration induced from the microlocal $V$-filtration on $R_F$. Hence
\(V_F^cA_F=q_F(\widetilde V_F^cR_F).\)
Using \eqref{eq:microTS}, we obtain
\[
 V_F^cA_F
 =\sum_{a+b\ge c}
 q_F\bigl((\widetilde V_f^aR_f)R_F(\widetilde V_g^bR_g)R_F\bigr).
\]
We identify the image of each $(a,b)$-term. Under the canonical isomorphism
\eqref{eq:AF-tensor},
\(q_F\bigl((\widetilde V_f^aR_f)R_F\bigr) =V_f^aA_f\otimes A_g,\)
and similarly
\(q_F\bigl((\widetilde V_g^bR_g)R_F\bigr) =A_f\otimes V_g^bA_g.\)
Indeed, the first extended ideal is generated by elements coming from
$\widetilde V_f^aR_f$, whose images in $A_f$ are precisely $V_f^aA_f$, and the second
statement is analogous. Since $q_F$ is an algebra homomorphism, the image of the product
of the two extended ideals is the product of their images. Thus
\[
 \begin{aligned}
 &q_F\bigl((\widetilde V_f^aR_f)R_F(\widetilde V_g^bR_g)R_F\bigr)\\
 &\qquad=(V_f^aA_f\otimes A_g)(A_f\otimes V_g^bA_g)
 =V_f^aA_f\otimes V_g^bA_g.
 \end{aligned}
\]
Substituting this into the preceding sum gives \eqref{eq:MilnorConv}. Thus, under
$A_F\simeq A_f\otimes_{\C}A_g$, the filtration $V_F$ is exactly the convolution
filtration of $(A_f,V_f)$ and $(A_g,V_g)$. Applying \cref{lem:conv-gr}, we have
\eqref{eq:MilnorConv-gr}.
\end{proof}

For $h=f,g,F$, the quotient filtration induced by \eqref{eq:V-A} on
$T_h=A_h/[h]A_h$ agrees, at every positive rational level, with the JKSY filtration
$V_{\HI}$ used in \cref{sec:JKSY}. Indeed, write
\(c=p+\lambda\), where \(p\in\Z_{\ge0}\) and \(\lambda\in\Q\cap(0,1]\).
The Musta\c{t}\u{a}--Popa comparison \cite[Theorem A]{MP-Q} shows that both filtrations are
the image in $T_h$ of the same microlocal level. Equivalently, for $c\in\Q_{>0}$,
\begin{equation}\label{eq:common-T-filtration}
 V_{\HI}^cT_h
 =\frac{\widetilde V_h^cR_h+(h,J_h)}{(h,J_h)}.
\end{equation}
From this point on, $V_h$ on $T_h$ denotes this common quotient filtration; when $V_h$ is applied to $A_h$, it continues to mean the quotient filtration defined in \eqref{eq:V-A}. No change of spectral convention is involved.

We now pass from the numerical defect to its filtered refinement. We retain
\(A=A_f\otimes A_g\), \(u=[f]\otimes1\), and \(v=1\otimes[g]\),
and define
\(K_{f,g}:=\frac{uA+vA}{(u+v)A}.\)
It is an ideal of $T_F=A/(u+v)A$. Its filtration is always the induced filtration
\begin{equation}\label{eq:Kfiltration}
 V^cK_{f,g}=K_{f,g}\cap V_F^cT_F.
\end{equation}

\begin{theorem}\label{thm:filtered-defect}
There is a canonical strict exact sequence
\begin{equation}\label{eq:strictK}
 0\longrightarrow K_{f,g}\longrightarrow T_F
 \xrightarrow{\pi}T_f\otimes_{\C}T_g\longrightarrow0,
\end{equation}
where the last term carries the convolution of its two quotient filtrations. Define
\(\defect_{f,g}(t):=\sum_c\dim_{\C}\Gr_V^cK_{f,g}\,t^c,\)
then we have
\begin{equation}\label{eq:filtered-defect}
 \Sp_F^{\Tj}(t)=\Sp_f^{\Tj}(t)\Sp_g^{\Tj}(t)+\defect_{f,g}(t).
\end{equation}
Its coefficients are nonnegative integers, and
\begin{equation}\label{eq:E1}
 \defect_{f,g}(1)=\eps_{f,g}=\sum_{j\ge2}b_j(f)b_j(g).
\end{equation}
The same exact sequence also yields
\begin{equation}\label{eq:spectrum-sandwich}
 \Sp_f^{\Tj}\Sp_g^{\Tj}\preceq\Sp_F^{\Tj}\preceq\Sp_f\Sp_g.
\end{equation}
The polynomial $\defect_{f,g}$ is zero if and only if $f\in J_f$ or $g\in J_g$. If
$a_h=\min_i\alpha^{\St}_{h,i}$, then
\begin{equation}\label{eq:support-bound}
 \Gr_V^cK_{f,g}=0\qquad(c<a_f+a_g+1).
\end{equation}
\end{theorem}
\begin{proof}
Since $(u+v)A\subseteq uA+vA$, the natural quotient map
\(A/(u+v)A\longrightarrow A/(uA+vA)\)
is surjective. The target is canonically
\(A/(uA+vA) \simeq (A_f/[f]A_f)\otimes(A_g/[g]A_g) =T_f\otimes T_g,\)
and the kernel is precisely
\(\frac{uA+vA}{(u+v)A}=K_{f,g}.\)
This gives the underlying exact sequence in \eqref{eq:strictK}. We now verify the
strictness of both maps.

By \cref{prop:MilnorConv}, $A=A_f\otimes A_g$ carries the convolution filtration
\(V^cA=\sum_{a+b\ge c}V_f^aA_f\otimes V_g^bA_g.\)
The maps $A_f\twoheadrightarrow T_f$ and $A_g\twoheadrightarrow T_g$ are quotient maps,
and hence strict for their quotient filtrations. Applying \cref{lem:tensor-strict} twice gives
\begin{equation}\label{eq:image-to-product}
 \im\bigl(V^cA\longrightarrow T_f\otimes T_g\bigr)
 =\sum_{a+b\ge c}V_f^aT_f\otimes V_g^bT_g
 =V^c(T_f\otimes T_g).
\end{equation}
On the other hand, $T_F=A/(u+v)A$ also carries the quotient filtration from $A$, so
\(V_F^cT_F=\im\bigl(V^cA\longrightarrow T_F\bigr).\)
For the composite $A\to T_F\xrightarrow{\pi}T_f\otimes T_g$, equation
\eqref{eq:image-to-product} therefore yields
\[
 \pi(V_F^cT_F)
 =\im\bigl(V^cA\longrightarrow T_f\otimes T_g\bigr)
 =V^c(T_f\otimes T_g).
\]
Thus $\pi$ is strict. The injection $K_{f,g}\hookrightarrow T_F$ is strict because the
filtration on $K_{f,g}$ is, by definition,
\(V^cK_{f,g}=K_{f,g}\cap V_F^cT_F.\)
Hence \eqref{eq:strictK} is a strict exact sequence.

Applying \cref{lem:strict-exact} gives, for every $c$,
\[
 0\longrightarrow\Gr_V^cK_{f,g}
 \longrightarrow\Gr_{V_F}^cT_F
 \longrightarrow\Gr_V^c(T_f\otimes T_g)
 \longrightarrow0.
\]
By \cref{lem:conv-gr},
\[
 \Gr_V^c(T_f\otimes T_g)
 \simeq
 \bigoplus_{a+b=c}\Gr_{V_f}^aT_f\otimes\Gr_{V_g}^bT_g.
\]
Taking dimensions therefore gives
\begin{equation}\label{eq:graded-dimension-defect}
 \dim_{\C}\Gr_{V_F}^cT_F
 =\dim_{\C}\Gr_V^cK_{f,g}
 +\sum_{a+b=c}
   \dim_{\C}\Gr_{V_f}^aT_f\,
   \dim_{\C}\Gr_{V_g}^bT_g.
\end{equation}
Multiplying by $t^c$ and summing over $c$ turns the three terms in
\eqref{eq:graded-dimension-defect} into, respectively,
\[
 \Sp_F^{\Tj}(t),
 \qquad
 \defect_{f,g}(t),
 \qquad
 \Sp_f^{\Tj}(t)\Sp_g^{\Tj}(t),
\]
so \eqref{eq:filtered-defect} follows.

Setting $t=1$ gives
\[
 \defect_{f,g}(1)
 =\sum_c\dim_{\C}\Gr_V^cK_{f,g}
 =\dim_{\C}K_{f,g}.
\]
The ordinary dimensions in \eqref{eq:strictK} also give
\(\dim_{\C}K_{f,g} =\tau_F-\tau_f\tau_g =\eps_{f,g}.\)
Combining this with \cref{prop:JordanDefect} yields
\(\defect_{f,g}(1)=\eps_{f,g} =\sum_{j\ge2}b_j(f)b_j(g),\)
which is \eqref{eq:E1}. In particular,
\(\defect_{f,g}=0 \iff K_{f,g}=0 \iff \eps_{f,g}=0.\)
The last condition is equivalent, by \cref{prop:JordanDefect} and K.~Saito's criterion,
to $f\in J_f$ or $g\in J_g$. The left inequality in
\eqref{eq:spectrum-sandwich} follows from the nonnegative coefficients of
$\defect_{f,g}$. For the right inequality, the strict quotient
$A_F\twoheadrightarrow T_F$ can only decrease the dimension of each associated-graded
piece, while the classical Thom--Sebastiani formula gives $\Sp_F=\Sp_f\Sp_g$.

It remains to prove the support bound. We first make explicit the identification used here. If $h$ is a germ in $d$ variables and $d\bz=dz_1\wedge\cdots\wedge dz_d$, then
\(dh\wedge\Omega^{d-1}=J_h\,d\bz,\)
so multiplication by the volume form induces a canonical vector-space isomorphism
\[
 A_h=R_h/J_h\xrightarrow{\sim}\Omega_h^d,
 \qquad [g]\longmapsto[g\,d\bz].
\]
Under this identification, multiplication by $[h]$ raises the quotient $V$-filtration on $A_h$ by at least one:
\begin{equation}\label{eq:Nh-raises}
 N_h(V_h^aA_h)\subseteq V_h^{a+1}A_h.
\end{equation}
Indeed, this is exactly the inclusion
\(hV^a\Omega_h^d\subseteq V^{a+1}\Omega_h^d\)
of \cite[Remark 1.1(ii), (1.10)]{JKSY}. Moreover, by \eqref{eq:St-mult} and the minimality of
\(a_h=\min_i\alpha^{\St}_{h,i},\)
all nonzero associated-graded pieces of $A_h$ occur in levels $\ge a_h$. Since the filtration is decreasing and exhaustive, this gives
\(A_h=V_h^{a_h}A_h.\)
To see explicitly what this means on the tensor product, take
$x\in V_f^{a_f}A_f$ and $y\in V_g^{a_g}A_g$. Then
\([f]x\in V_f^{a_f+1}A_f,\)
so, for $u=[f]\otimes1$,
\[
 u(x\otimes y)=[f]x\otimes y
 \in V_f^{a_f+1}A_f\otimes V_g^{a_g}A_g
 \subseteq V^{a_f+a_g+1}A.
\]
The same argument gives
\(v(x\otimes y)=x\otimes[g]y \in V^{a_f+a_g+1}A.\)
Since such tensors span $A=A_f\otimes A_g$, we obtain
\(uA\subseteq V^{a_f+a_g+1}A\) and \(vA\subseteq V^{a_f+a_g+1}A\),
and hence
\(uA+vA\subseteq V^{a_f+a_g+1}A.\)
The image of $uA+vA$ in $T_F$ is $K_{f,g}$, so
\(K_{f,g}\subseteq V_F^{a_f+a_g+1}T_F.\)
Using the induced filtration on the kernel,
\(V^{a_f+a_g+1}K_{f,g} =K_{f,g}\cap V_F^{a_f+a_g+1}T_F =K_{f,g}.\)
Therefore no associated-graded piece of $K_{f,g}$ can occur below this level, i.e.
\(\Gr_V^cK_{f,g}=0\) for \(c<a_f+a_g+1\),
which is \eqref{eq:support-bound}.
\end{proof}

\begin{remark}\label{rem:kernel-filtration}
The block lengths determine the total dimension of $K_{f,g}$. They do not specify its
filtration levels and are not used as filtered Jordan data. The proof therefore does not
replace $\Gr_V((u+v)A)$ by the image of $\Gr_V(u+v)$; such a replacement would require
an additional strictness assertion for multiplication by $u+v$.
\end{remark}

\begin{corollary}
\label{cor:ST-Tj-spectrum}
The Tjurina subspectrum satisfies the Sebastiani--Thom property,
namely
\[
\Sp_F^{\Tj}(t)=\Sp_f^{\Tj}(t)\Sp_g^{\Tj}(t),
\]
if and only if
\[
\tau_F=\tau_f\tau_g.
\]
In particular, this gives an affirmative answer to the question raised in
Almir\'on's Remark 3~\cite{Almiron}.
\end{corollary}
\begin{proof}
By \cref{thm:filtered-defect}, the difference between
\(\Sp_F^{\Tj}(t)\) and \(\Sp_f^{\Tj}(t)\Sp_g^{\Tj}(t)\) is the polynomial
\(\defect_{f,g}(t)\), whose coefficients are nonnegative and whose value at
one is \(\tau_F-\tau_f\tau_g\). Hence the product formula holds if and only if
\(\defect_{f,g}(t)=0\), which is equivalent to
\(\defect_{f,g}(1)=\tau_F-\tau_f\tau_g=0\).
\end{proof}

\begin{corollary}\label{cor:cutoff}
For $c\in\Q_{>0}$, let $H_h(c)$ be as in the introduction, and set
\(C_{f,g}(c):=\#\{(i,j):\alpha^{\Tj}_{f,i}+\alpha^{\Tj}_{g,j}<c\}.\)
Then
\begin{equation}\label{eq:cutoff}
 H_F(c)=C_{f,g}(c)+E_{f,g}(c),
 \qquad
 E_{f,g}(c):=\dim_{\C}K_{f,g}/V^cK_{f,g}.
\end{equation}
The error $E_{f,g}(c)$ is nondecreasing in $c$, takes integer values, and satisfies
\(0\le E_{f,g}(c)\le\eps_{f,g}.\)
It is zero for $c\le a_f+a_g+1$ and equals $\eps_{f,g}$ for all sufficiently large $c$.
\end{corollary}
\begin{proof}
Write
\[
 \defect_{f,g}(t)=\sum_{\gamma}e_\gamma t^\gamma,
 \qquad
 e_\gamma=\dim_{\C}\Gr_V^\gamma K_{f,g}.
\]
By \cref{thm:filtered-defect},
\[
 \Sp_F^{\Tj}(t)
 =\Sp_f^{\Tj}(t)\Sp_g^{\Tj}(t)
  +\sum_\gamma e_\gamma t^\gamma.
\]
The coefficient multiset of the product
$\Sp_f^{\Tj}(t)\Sp_g^{\Tj}(t)$ consists of all pairwise sums
$\alpha^{\Tj}_{f,i}+\alpha^{\Tj}_{g,j}$. Hence, after summing only the exponents
strictly below $c$, \cref{thm:counting} gives
\(H_F(c) =C_{f,g}(c)+\sum_{\gamma<c}e_\gamma.\)
Because the filtration on $K_{f,g}$ is decreasing, finite, exhaustive and separated,
its quotient by $V^cK_{f,g}$ is the sum of the graded pieces lying strictly below the
cutoff. Therefore
\[
 \sum_{\gamma<c}e_\gamma
 =\dim_{\C}\bigl(K_{f,g}/V^cK_{f,g}\bigr)
 =E_{f,g}(c),
\]
which proves \eqref{eq:cutoff}.

As $c$ increases, $V^cK_{f,g}$ decreases, so $E_{f,g}(c)$ is nondecreasing and integer
valued. Its minimum is $0$ and its maximum is $\dim_{\C}K_{f,g}=\eps_{f,g}$.
The support bound \eqref{eq:support-bound} gives $E_{f,g}(c)=0$ for
$c\le a_f+a_g+1$; equality at the endpoint is included because exponents exactly at the
cutoff are not counted. Finally, separatedness of the filtration gives $V^cK_{f,g}=0$ for
all sufficiently large $c$, and then $E_{f,g}(c)=\eps_{f,g}$.
\end{proof}

\begin{corollary}\label{cor:weightedTS}
Assume that $f$ is weighted homogeneous. Then
\(T_F\simeq A_f\otimes_{\C}T_g,\)
and under this isomorphism
\begin{equation}\label{eq:TjConvWeighted}
 V_F^cT_F=\sum_{a+b\ge c}V_f^aA_f\otimes V_g^bT_g.
\end{equation}
We have
\begin{equation}\label{eq:TjConvWeighted-gr}
 \Gr_{V_F}^cT_F\simeq
 \bigoplus_{a+b=c}\Gr_{V_f}^aA_f\otimes\Gr_{V_g}^bT_g,
\end{equation}
and
\begin{equation}\label{eq:weighted-spectrum-product}
 \Sp_F^{\Tj}(t)=\Sp_f(t)\Sp_g^{\Tj}(t)
 =\Sp_f^{\Tj}(t)\Sp_g^{\Tj}(t).
\end{equation}
The analogous statement holds with $f$ and $g$ interchanged.
\end{corollary}
\begin{proof}
Since $f$ is weighted homogeneous, $f\in J_f$. Hence
$[f]=0$ in $A_f$, so $T_f=A_f$ with the same quotient filtration. In the notation of
\cref{thm:filtered-defect}, this gives $u=[f]\otimes1=0$, and therefore
\(K_{f,g}=\frac{uA+vA}{(u+v)A}=0.\)
The strict exact sequence \eqref{eq:strictK} consequently reduces to a strict isomorphism
\(T_F\xrightarrow{\sim}T_f\otimes_{\C}T_g =A_f\otimes_{\C}T_g,\)
where the right-hand side carries the convolution filtration. Thus
\(V_F^cT_F =\sum_{a+b\ge c}V_f^aA_f\otimes V_g^bT_g,\)
which is \eqref{eq:TjConvWeighted}. Applying \cref{lem:conv-gr} gives
\eqref{eq:TjConvWeighted-gr}; taking Hilbert--Poincar\'e series yields
\(\Sp_F^{\Tj}(t)=\Sp_f^{\Tj}(t)\Sp_g^{\Tj}(t).\)
Finally, weighted homogeneity gives $\Sp_f^{\Tj}=\Sp_f$, proving
\eqref{eq:weighted-spectrum-product}. The argument with $f$ and $g$ interchanged is identical.
\end{proof}

\begin{corollary}\label{cor:weighted-mk}
Under the hypotheses of \cref{cor:weightedTS}, for every $k\in\Z_{\ge0}$,
\begin{equation}\label{eq:weighted-mk}
 m_k(F)=\#\{(i,j):\alpha^{\St}_{f,i}+\alpha^{\Tj}_{g,j}<k+1\}.
\end{equation}
Equivalently, since $f$ is weighted homogeneous,
\(m_k(F)=\#\{(i,j):\alpha^{\Tj}_{f,i}+\alpha^{\Tj}_{g,j}<k+1\}.\)
\end{corollary}
\begin{proof}
By \cref{cor:weightedTS}, the Tjurina spectral multiset of $F$ is the multiset of
all sums
\(\alpha^{\St}_{f,i}+\alpha^{\Tj}_{g,j}.\)
Apply \cref{cor:mk-count} to $F$.
\end{proof}

\begin{remark}\label{rem:fractional-info}
\cref{cor:weighted-mk} shows precisely which information is needed for an exact
Thom--Sebastiani computation: one needs the locations of the Tjurina spectral exponents
inside the unit intervals, not merely their cumulative counts at integer cutoffs.
\cref{sec:bounds} quantifies the information that survives after those fractional
positions are forgotten.
\end{remark}

\begin{example}\label{ex:nonzero-defect}
Let
\[
 h=x^5+y^5+x^3y^3,
 \qquad
 P(t)=\bigl(t^{1/5}+t^{2/5}+t^{3/5}+t^{4/5}\bigr)^2.
\]
Then
\begin{equation}\label{eq:example-spectra}
 \Sp_h=P,
 \qquad
 \Sp_h^{\Tj}=P-t^{8/5},
 \qquad
 \Sp_{h\boxplusop h}^{\Tj}=(P-t^{8/5})^2+t^2.
\end{equation}
We obtain $\mu_h=16$, $\tau_h=15$, and
\(\tau_{h\boxplusop h}=226>225=\tau_h^2.\)

The germ is semi-weighted-homogeneous with
principal part $x^5+y^5$ and weights $(1/5,1/5)$. By \cite[\S1.6]{JKSY}, its spectrum is
$P$ and the quotient $V$-filtration is induced by the weighted filtration. In $A_h$, the
derivative relations are
\(x^4=-\frac35x^2y^3\) and \(y^4=-\frac35x^3y^2\).
Every monomial of total degree seven is divisible by $x^4$ or $y^4$. Replacing one
factor $x^4$ or $y^4$ by the relations above raises the total degree by one, so
\(\mathfrak m^7\subseteq\mathfrak m^8.\)
The reverse inclusion is automatic, hence $\mathfrak m^7=\mathfrak m^8=\mathfrak m\mathfrak m^7$. So by Nakayama's lemma, we have
$\mathfrak m^7=0$.

Reduction by the two derivative relations now shows that the sixteen classes
\(x^iy^j\), with \(0\le i,j\le3\),
span $A_h$. The principal weighted-homogeneous part $x^5+y^5$ has Milnor number
$(5-1)^2=16$, and semi-weighted homogeneity preserves the Milnor number, so
$\mu_h=16$. The sixteen spanning classes are therefore a basis. They are adapted to $V$,
with levels
\(\frac{i+j+2}{5},\)
as follows from the weighted filtration and its Hilbert series $P$.

For the class of the equation, write
\(h=h_5+h_6\), where \(h_5=x^5+y^5\) and \(h_6=x^3y^3\).
Euler's identities for the two homogeneous pieces are
\(x(h_5)_x+y(h_5)_y=5h_5\) and \(x(h_6)_x+y(h_6)_y=6h_6\).
Since $h_x=h_y=0$ in the Milnor algebra, the sum of these identities gives
\(5[h_5]+6[h_6]=0.\)
As $[h]=[h_5]+[h_6]$, it follows that
\(5[h]=-[h_6]=-[x^3y^3].\)
Set
\(s=[x^3y^3]\ne0.\)
The class $s$ is one of the basis monomials above, and $s\mathfrak m=0$ because every
product of $s$ with $x$ or $y$ has total degree seven and $\mathfrak m^7=0$. Therefore
\[
 [h]=-\frac15s,
 \qquad
 N_h(1)=-\frac15s,
 \qquad
 N_h(\mathfrak m)=0.
\]
Thus the Jordan chain generated by $1$ is a single block of length two,
\(1\longmapsto -\frac15s\longmapsto0,\)
and the remaining fourteen basis directions give blocks of length one. Hence
\[
 \tau_h=15,
 \qquad
 b_1(h)=15,
 \qquad
 b_2(h)=1,
 \qquad
 b_j(h)=0\quad(j\ge3).
\]
The image $[h]A_h=\C s$ has level
\(\frac{3+3+2}{5}=\frac85.\)
Since $T_h=A_h/[h]A_h$, the quotient removes precisely this one graded class, and hence
\(\Sp_h^{\Tj}(t)=P(t)-t^{8/5}.\)
Moreover,
\(\eps_{h,h}=\sum_{j\ge2}b_j(h)^2=1,\)
so $\dim_{\C}K_{h,h}=1$ and
\(\tau_{h\boxplusop h}=\tau_h^2+1=225+1=226.\)

The unique nonzero action of $N_h$ sends a class of level $2/5$ to one of level $8/5$,
so, in the adapted tensor basis, $N_h\otimes1+1\otimes N_h$ is homogeneous of degree
\(\frac85-\frac25=\frac65.\)
Now
\[
\begin{aligned}
 K_{h,h}&=\frac{uA+vA}{(u+v)A},
 &\qquad A&=A_h\otimes A_h,\\
 u&=[h]\otimes1=-\frac15s\otimes1,
 &\qquad v&=1\otimes[h]=-\frac15 1\otimes s.
\end{aligned}
\]
Because $\dim_{\C}K_{h,h}=1$, it is enough to exhibit one nonzero class. In the quotient,
$(u+v)(1\otimes1)=0$ gives
\(s\otimes1+1\otimes s=0,\)
so $[s\otimes1]=-[1\otimes s]$ and either class can represent the unique generator.
Its level is
\(\frac85+\frac25=2.\)
To see that it is nonzero, note that the level-$2$ part of the image of
$N_h\otimes1+1\otimes N_h$ is spanned by
\(s\otimes1+1\otimes s,\)
whereas every other nonzero image has larger level. Thus the quotient retains a
one-dimensional class at level $2$. Therefore
\(\Gr_V^2K_{h,h}\simeq\C\) and \(\Gr_V^cK_{h,h}=0\) for \(c\ne2\),
and hence
\(\defect_{h,h}(t)=t^2.\)
This extra exponent is not counted by $H_{h\boxplusop h}(2)$; it enters the count
at cutoffs strictly greater than $2$.
\end{example}

\section{Bounds, reconstruction, and symmetry}\label{sec:bounds}

The cutoff formula is used below at successively finer resolutions, starting with integer bins and ending with rational grids. This separates the loss of fractional position inside a bin from the genuine filtered contribution of $K_{f,g}$.

\subsection{Integer-bin bounds and the boundary term}
Put $m_{-1}(h)=0$ and, for $r\in\Z_{\ge0}$, define
\begin{equation}\label{eq:Delta}
 \Delta_r(h):=m_r(h)-m_{r-1}(h)
 =\#\{i:r\le\alpha^{\Tj}_{h,i}<r+1\}.
\end{equation}
For $k\ge0$, set
\begin{equation}\label{eq:LU}
 L_k(f,g):=\sum_{r+s\le k-1}\Delta_r(f)\Delta_s(g),
 \qquad
 U_k(f,g):=\sum_{r+s\le k}\Delta_r(f)\Delta_s(g),
\end{equation}
with empty sums equal to zero.

\begin{theorem}
\label{thm:integer-general}
For arbitrary $f,g,F$ as in \cref{sec:ST}, and for $c\in\Q_{>0}$, set
\[
C_{f,g}(c)
:=
\#\bigl\{(i,j):
\alpha^{\Tj}_{f,i}+\alpha^{\Tj}_{g,j}<c
\bigr\}.
\]
Then, for every $k\ge0$,
\begin{equation}\label{eq:integer-general}
L_k(f,g)\le m_k(F)\le U_k(f,g)+\eps_{f,g}.
\end{equation}
More precisely, define
\begin{equation}\label{eq:Bpair}
B_k^{\mathrm{pair}}(f,g):=
\#\Bigl\{(i,j):
\lfloor\alpha^{\Tj}_{f,i}\rfloor
+\lfloor\alpha^{\Tj}_{g,j}\rfloor=k,\ 
\alpha^{\Tj}_{f,i}+\alpha^{\Tj}_{g,j}<k+1
\Bigr\}.
\end{equation}
Then
\begin{equation}\label{eq:integer-exact}
m_k(F)
=
L_k(f,g)
+
B_k^{\mathrm{pair}}(f,g)
+
E_{f,g}(k+1),
\end{equation}
where
\begin{equation}\label{eq:Bpair-bound}
0\le B_k^{\mathrm{pair}}(f,g)
\le
\sum_{r+s=k}\Delta_r(f)\Delta_s(g).
\end{equation}
For a pair in a boundary block $r+s=k$, write
\[
\alpha^{\Tj}_{f,i}=r+\theta,
\qquad
\alpha^{\Tj}_{g,j}=s+\eta,
\qquad
0\le\theta,\eta<1.
\]
It contributes to $B_k^{\mathrm{pair}}(f,g)$ if and only if
\begin{equation}\label{eq:thetaeta}
\theta+\eta<1.
\end{equation}
If $k+1\le a_f+a_g+1$, then $E_{f,g}(k+1)=0$. For two reduced
isolated plane curve factors,
\(\eps_{f,g}=(\mu_f-\tau_f)(\mu_g-\tau_g).\)
\end{theorem}

\begin{proof}
Partition the pairs
$\bigl(\alpha^{\Tj}_{f,i},\alpha^{\Tj}_{g,j}\bigr)$
into the unit blocks
\[
\mathcal B_{r,s}
:=
\bigl\{(i,j):
r\le\alpha^{\Tj}_{f,i}<r+1,\ 
s\le\alpha^{\Tj}_{g,j}<s+1
\bigr\}.
\]
By the definition of $\Delta_r$, $
\#\mathcal B_{r,s}=\Delta_r(f)\Delta_s(g).$

If $r+s\le k-1$, then for every $(i,j)\in\mathcal B_{r,s}$ we have
\[
\alpha^{\Tj}_{f,i}+\alpha^{\Tj}_{g,j}
<
(r+1)+(s+1)
=
r+s+2
\le k+1.
\]
Thus every pair in such a block contributes to $C_{f,g}(k+1)$.
The total contribution of these blocks is therefore $L_k(f,g)=\sum_{r+s\le k-1}\Delta_r(f)\Delta_s(g).$

If $r+s\ge k+1$, then every $(i,j)\in\mathcal B_{r,s}$ satisfies
$\alpha^{\Tj}_{f,i}+\alpha^{\Tj}_{g,j}\ge r+s\ge k+1,$
so no pair in such a block contributes. Hence only the boundary
blocks with $r+s=k$ can be cut by the line
$\alpha+\beta=k+1$.

For a pair in a boundary block, write
$\alpha^{\Tj}_{f,i}=r+\theta$ and
$\alpha^{\Tj}_{g,j}=s+\eta$, with
$0\le\theta,\eta<1$. Since $r+s=k$, we have
\[
\alpha^{\Tj}_{f,i}+\alpha^{\Tj}_{g,j}
=
k+\theta+\eta.
\]
Consequently, such a pair contributes to $C_{f,g}(k+1)$ exactly when
$k+\theta+\eta<k+1$, or equivalently when
\eqref{eq:thetaeta} holds. By the definition of
$B_k^{\mathrm{pair}}(f,g)$, this gives
\[
C_{f,g}(k+1)
=
L_k(f,g)+B_k^{\mathrm{pair}}(f,g).
\]

The boundary blocks contain altogether $\sum_{r+s=k}\Delta_r(f)\Delta_s(g)$ pairs. Hence
\[
0\le B_k^{\mathrm{pair}}(f,g)
\le
\sum_{r+s=k}\Delta_r(f)\Delta_s(g),
\]
which is \eqref{eq:Bpair-bound}. Since
\[
U_k(f,g)-L_k(f,g)
=
\sum_{r+s=k}\Delta_r(f)\Delta_s(g),
\]
we also obtain
\[
L_k(f,g)
\le
C_{f,g}(k+1)
\le
U_k(f,g).
\]

Now apply the cutoff formula \eqref{eq:cutoff}. Since
$m_k(F)=H_F(k+1)$, we have
\[
m_k(F)
=
C_{f,g}(k+1)+E_{f,g}(k+1).
\]
Substituting the preceding decomposition of $C_{f,g}(k+1)$ gives
\[
m_k(F)
=
L_k(f,g)
+
B_k^{\mathrm{pair}}(f,g)
+
E_{f,g}(k+1),
\]
which is \eqref{eq:integer-exact}. Finally,
\[
0\le B_k^{\mathrm{pair}}(f,g)
\le U_k(f,g)-L_k(f,g)
\]
and
\[
0\le E_{f,g}(k+1)\le\eps_{f,g},
\]
so
\[
L_k(f,g)
\le
m_k(F)
\le
U_k(f,g)+\eps_{f,g}.
\]
This proves \eqref{eq:integer-general}. The low-cutoff assertion follows
from \cref{cor:cutoff}, and the plane-curve formula follows from
\cref{cor:squarezero}.
\end{proof}

When $f\in J_f$ or $g\in J_g$, we have $K_{f,g}=0$, so $E_{f,g}\equiv0$ and \eqref{eq:integer-general} becomes the pure integer-data bound. The block sums may also be written as
\[
 L_k(f,g)=\sum_{r=0}^{k-1}\bigl(m_r(f)-m_{r-1}(f)\bigr)m_{k-1-r}(g),
\]
\(U_k(f,g)=\sum_{r=0}^{k}\bigl(m_r(f)-m_{r-1}(f)\bigr)m_{k-r}(g).\)
Geometrically, the integer data divide the $(\alpha,\beta)$-plane into unit squares. All squares with $r+s\le k-1$ are completely contributing, all squares with $r+s\ge k+1$ are noncontributing, and the diagonal squares $r+s=k$ contain exactly the missing fractional-position information measured by $B_k^{\mathrm{pair}}$.

\begin{remark}\label{rem:two-discrepancies}
The term $B_k^{\mathrm{pair}}$ measures positions inside unit intervals, while $E_{f,g}(k+1)$ measures the filtered kernel $K_{f,g}$. The first may be nonzero even for weighted-homogeneous factors; the second vanishes at every cutoff if $f\in J_f$ or $g\in J_g$.
\end{remark}

\subsection{The \texorpdfstring{$\Q$}{Q}-divisor version}
For $\lambda\in\Q\cap(0,1]$, define
\[
 m_{k,\lambda}(h):=\dim_{\C}\frac{R_h}{I_k(\lambda Z_h)+(h,J_h)},
 \qquad k\ge0,
\]
and put $m_{-1,\lambda}(h)=0$. By \cref{thm:counting},
\begin{equation}\label{eq:mlambda}
 m_{k,\lambda}(h)=\#\{i:\alpha^{\Tj}_{h,i}<k+\lambda\}.
\end{equation}
For $r\ge0$, set
\(\Delta_r^{(\lambda)}(h):=m_{r,\lambda}(h)-m_{r-1,\lambda}(h).\)
Then
\begin{equation}\label{eq:Delta-lambda}
 \Delta_r^{(\lambda)}(h)
 =\#\{i:r-1+\lambda\le\alpha^{\Tj}_{h,i}<r+\lambda\}.
\end{equation}

\begin{theorem}\label{thm:Qbounds}
For arbitrary $f,g,F$ as in \cref{sec:ST}, let
\[
 \lambda,\lambda_1,\lambda_2\in\Q\cap(0,1],
 \qquad \lambda_1+\lambda_2=\lambda+1.
\]
For $k\ge0$, set
\begin{equation}\label{eq:Q-LU}
\begin{aligned}
 L_k^{(\lambda_1,\lambda_2)}(f,g)
 &:=\sum_{r+s\le k-1}\Delta_r^{(\lambda_1)}(f)\Delta_s^{(\lambda_2)}(g),\\
 U_k^{(\lambda_1,\lambda_2)}(f,g)
 &:=\sum_{r+s\le k}\Delta_r^{(\lambda_1)}(f)\Delta_s^{(\lambda_2)}(g).
\end{aligned}
\end{equation}
Define the shifted boundary term
\begin{equation}\label{eq:Q-Bpair}
\begin{aligned}
 B_{k}^{(\lambda_1,\lambda_2),\mathrm{pair}}(f,g)
 :=\#\Bigl\{(i,j):\;&\text{for some }r,s\ge0\text{ with }r+s=k,\\
 &r-1+\lambda_1\le\alpha^{\Tj}_{f,i}<r+\lambda_1,\\
 &s-1+\lambda_2\le\alpha^{\Tj}_{g,j}<s+\lambda_2,\\
 &\alpha^{\Tj}_{f,i}+\alpha^{\Tj}_{g,j}<k+\lambda\Bigr\}.
\end{aligned}
\end{equation}
Then
\begin{equation}\label{eq:Q-exact}
 m_{k,\lambda}(F)
 =L_k^{(\lambda_1,\lambda_2)}(f,g)
 +B_{k}^{(\lambda_1,\lambda_2),\mathrm{pair}}(f,g)
 +E_{f,g}(k+\lambda),
\end{equation}
and hence
\begin{equation}\label{eq:Qbounds}
 L_k^{(\lambda_1,\lambda_2)}(f,g)
 \le m_{k,\lambda}(F)
 \le U_k^{(\lambda_1,\lambda_2)}(f,g)+E_{f,g}(k+\lambda)
 \le U_k^{(\lambda_1,\lambda_2)}(f,g)+\eps_{f,g}.
\end{equation}
Moreover,
\begin{equation}\label{eq:Q-Bpair-bound}
 0\le B_{k}^{(\lambda_1,\lambda_2),\mathrm{pair}}(f,g)
 \le\sum_{r+s=k}\Delta_r^{(\lambda_1)}(f)\Delta_s^{(\lambda_2)}(g)
 =U_k^{(\lambda_1,\lambda_2)}(f,g)-L_k^{(\lambda_1,\lambda_2)}(f,g).
\end{equation}
For a boundary block $r+s=k$, write
\[
 \alpha^{\Tj}_{f,i}=r-1+\lambda_1+\theta,
 \qquad
 \alpha^{\Tj}_{g,j}=s-1+\lambda_2+\eta,
 \qquad 0\le\theta,\eta<1.
\]
Such a pair contributes to the boundary term if and only if
\begin{equation}\label{eq:Q-thetaeta}
 \theta+\eta<1.
\end{equation}
If $k+\lambda\le a_f+a_g+1$, then $E_{f,g}(k+\lambda)=0$, so the exact formula
\eqref{eq:Q-exact} and the bounds \eqref{eq:Qbounds} are defect-free at that cutoff.
\end{theorem}
\begin{proof}
By \eqref{eq:mlambda} and the cutoff formula \eqref{eq:cutoff},
\[
 m_{k,\lambda}(F)
 =H_F(k+\lambda)
 =C_{f,g}(k+\lambda)+E_{f,g}(k+\lambda).
\]
Partition the pairwise sums of the two Tjurina spectra into the shifted blocks
\[
 r-1+\lambda_1\le\alpha^{\Tj}_{f,i}<r+\lambda_1,
 \qquad
 s-1+\lambda_2\le\alpha^{\Tj}_{g,j}<s+\lambda_2.
\]
By \eqref{eq:Delta-lambda}, the size of the $(r,s)$-block is
$\Delta_r^{(\lambda_1)}(f)\Delta_s^{(\lambda_2)}(g)$. Since
$\lambda_1+\lambda_2=\lambda+1$, every pair in this block has sum in
\([r+s+\lambda-1,\ r+s+\lambda+1).\)
If $r+s\le k-1$, the whole block lies strictly below $k+\lambda$ and therefore contributes.
If $r+s\ge k+1$, the whole block lies at or above $k+\lambda$ and contributes nothing.
Thus only the boundary blocks $r+s=k$ are cut by the line
$\alpha+\beta=k+\lambda$.

On such a boundary block,
\(\alpha^{\Tj}_{f,i}+\alpha^{\Tj}_{g,j} =k+\lambda-1+\theta+\eta,\)
so the strict cutoff condition is equivalent to $\theta+\eta<1$. It follows that
\[
 C_{f,g}(k+\lambda)
 =L_k^{(\lambda_1,\lambda_2)}(f,g)
 +B_{k}^{(\lambda_1,\lambda_2),\mathrm{pair}}(f,g),
\]
which together with the cutoff formula proves \eqref{eq:Q-exact}. The boundary term is
nonnegative and is bounded by the total size of the boundary blocks, giving
\eqref{eq:Q-Bpair-bound}. Combining this with
$0\le E_{f,g}(k+\lambda)\le\eps_{f,g}$ then proves \eqref{eq:Qbounds}. Finally,
\cref{cor:cutoff} gives $E_{f,g}(k+\lambda)=0$ whenever
$k+\lambda\le a_f+a_g+1$.
\end{proof}

\begin{corollary}\label{cor:Qbounds-defectfree}
Assume the hypotheses of \cref{cor:weightedTS}. With
$\lambda,\lambda_1,\lambda_2$ as in \cref{thm:Qbounds}, one has, for every $k\ge0$,
\begin{equation}\label{eq:Qbounds-defectfree}
 \sum_{r+s\le k-1}\Delta_r^{(\lambda_1)}(f)\Delta_s^{(\lambda_2)}(g)
 \le m_{k,\lambda}(F)
 \le
 \sum_{r+s\le k}\Delta_r^{(\lambda_1)}(f)\Delta_s^{(\lambda_2)}(g).
\end{equation}
More precisely,
\[
 m_{k,\lambda}(F)
 =L_k^{(\lambda_1,\lambda_2)}(f,g)
 +B_{k}^{(\lambda_1,\lambda_2),\mathrm{pair}}(f,g),
\]
and the difference between the upper and lower bounds in
\eqref{eq:Qbounds-defectfree} is
\[
 \sum_{r+s=k}\Delta_r^{(\lambda_1)}(f)\Delta_s^{(\lambda_2)}(g).
\]
For a boundary block $r+s=k$, the contribution criterion is again
$\theta+\eta<1$ in the notation of \cref{thm:Qbounds}. The analogous statement holds
with $f$ and $g$ interchanged.
\end{corollary}
\begin{proof}
Under the hypotheses of \cref{cor:weightedTS}, one has $K_{f,g}=0$, hence
$E_{f,g}(c)=0$ for every cutoff $c$. The result follows immediately from
\cref{thm:Qbounds}.
\end{proof}

When $\lambda=1$, the condition $\lambda_1+\lambda_2=2$ forces
$\lambda_1=\lambda_2=1$. In this case, \cref{thm:Qbounds} recovers the arbitrary-factor
integer-bin theorem \cref{thm:integer-general}, while
\cref{cor:Qbounds-defectfree} is its defect-free specialization.

\subsection{Finite rational sampling and finer grids}
The integer sequence records only counts in unit intervals. Rational Hodge--Tjurina numbers give an exact finite remedy once a common denominator is known.

\begin{proposition}\label{prop:rational-recovery}
Let $h$ be an isolated hypersurface singularity in $d$ variables with $d\ge2$. Assume its Tjurina exponents belong to $D^{-1}\Z$, where $D$ is a positive integer. Set
\begin{equation}\label{eq:qD}
 q_a^{(D)}(h):=H_h\left(\frac{a+1}{D}\right)-H_h\left(\frac aD\right),
 \qquad 0\le a\le dD-1.
\end{equation}
Set $q_a^{(D)}(h)=0$ for $a\notin\{0,\ldots,dD-1\}$; this convention is needed in the convolution sum below when $q$ is large.
Then
\begin{equation}\label{eq:rational-recovery}
 \Sp_h^{\Tj}(t)=\sum_{a=0}^{dD-1}q_a^{(D)}(h)t^{a/D}.
\end{equation}
Thus finitely many Hodge--Tjurina numbers with $\lambda\in\{1/D,\ldots,D/D\}$ determine the entire Tjurina spectrum. If both factors satisfy the denominator condition and ($f\in J_f$ or $g\in J_g$), then for every integer $q>0$,
\begin{equation}\label{eq:rational-product}
 H_F(q/D)=\sum_{a+b\le q-1}q_a^{(D)}(f)q_b^{(D)}(g).
\end{equation}
If both factors satisfy the same denominator condition but no weighted-homogeneity assumption is imposed, then the right-hand side is $C_{f,g}(q/D)$ and
\(H_F(q/D)=C_{f,g}(q/D)+E_{f,g}(q/D).\)
\end{proposition}
\begin{proof}
The Tjurina spectrum is a subspectrum of the Steenbrink spectrum \cite[Introduction, (0.4)]{JKSY}, whose exponents lie in $(0,d)$. By \eqref{eq:count-intro}, $q_a^{(D)}(h)$ counts exponents in $[a/D,(a+1)/D)$. Under the lattice assumption, this interval contains exactly one possible exponent, namely $a/D$, proving \eqref{eq:rational-recovery}. Every grid value is an allowed Hodge quotient: for $c=a/D>0$, take $p=\lceil c\rceil-1$ and $\lambda=c-p\in\{1/D,\ldots,1\}$. For lattice exponents,
\(\frac aD+\frac bD<\frac qD\iff a+b\le q-1.\)
The product case follows from \cref{thm:filtered-defect} and its vanishing criterion. Under the same denominator hypothesis, the final identity follows from \cref{cor:cutoff}.
\end{proof}

The denominator hypothesis can be dropped at the cost of two nested approximations.

\begin{proposition}\label{prop:refined-grid}
For any positive integer $D$, define $q_a^{(D)}(h)$ by \eqref{eq:qD}, extended by zero outside its finite range. For $c=q/D>0$, put
\[
 L_D(c)=\sum_{a+b\le q-2}q_a^{(D)}(f)q_b^{(D)}(g),
 \qquad
 U_D(c)=\sum_{a+b\le q-1}q_a^{(D)}(f)q_b^{(D)}(g).
\]
Then
\begin{equation}\label{eq:refined-bounds}
 L_D(c)\le C_{f,g}(c)\le U_D(c),
 \qquad
 L_D(c)\le H_F(c)\le U_D(c)+\eps_{f,g}.
\end{equation}
If the grid is subdivided from $D$ to $eD$ with $e\in\Z_{>0}$, then
\begin{equation}\label{eq:nested-grid}
 L_D(c)\le L_{eD}(c)\le C_{f,g}(c)\le U_{eD}(c)\le U_D(c).
\end{equation}
\end{proposition}
\begin{proof}
The number $q_a^{(D)}(h)$ counts exponents in the half-open interval $[a/D,(a+1)/D)$ with no lattice assumption. A pair in bins $(a,b)$ has sum in
\(\left[\frac{a+b}{D},\frac{a+b+2}{D}\right).\)
Bins with $a+b\le q-2$ therefore contribute completely, whereas any bin containing a contributing pair must satisfy $a+b\le q-1$. This establishes the first pair of inequalities. After subdivision, each old interval is the disjoint union of $e$ new intervals: a new block inside a completely contributing old block remains completely contributing, and every potentially contributing new block lies inside a potentially contributing old block. These inclusions give \eqref{eq:nested-grid} and the bounds for $H_F(c)$ then follow from \eqref{eq:cutoff}.
\end{proof}

\subsection{A symmetry bound for weighted-homogeneous self sums}
The preceding results concern the recovery of spectral information from rational cutoff data. We now turn to a complementary consequence of spectral symmetry in the weighted-homogeneous case, which yields a sharp bound for self sums.

\begin{proposition}\label{prop:central-cutoff}
Let $h$ be weighted homogeneous in $d$ variables, with isolated singularity and Milnor number $\mu$. For every spectral value $a$, let $n_a$ denote its multiplicity. Here we write $h\boxplusop h$ for the self sum in two disjoint variable sets. Then
\begin{equation}\label{eq:central-cutoff}
 m_{d-1}(h\boxplusop h)
 =\frac{\mu^2-\sum_a n_a^2}{2}
 \le\frac{\mu(\mu-1)}2.
\end{equation}
Equality in the upper bound holds if and only if all spectral exponents of $h$ are distinct.
\end{proposition}
\begin{proof}
Because $h$ is weighted homogeneous, its Tjurina and Steenbrink spectra agree and the
filtered defect vanishes. Hence the self-sum spectrum is the ordinary product
\(\Sp_{h\boxplusop h}^{\Tj}(t)=\Sp_h(t)^2.\)
Thus the $\mu^2$ spectral exponents of the self sum are precisely the ordered pairwise sums
of the $\mu$ spectral exponents of $h$.

The spectrum of $h$ is symmetric: $n_a=n_{d-a}$ \cite[(1.2)]{JKSY}. Equivalently, if
we list the spectral exponents in nondecreasing order
\(\alpha_1\le\alpha_2\le\cdots\le\alpha_\mu,\)
then, with multiplicity,
\begin{equation}\label{eq:ordered-symmetry}
 \alpha_i+\alpha_{\mu+1-i}=d,
 \qquad 1\le i\le\mu.
\end{equation}
This makes the symmetry of the ordered pairwise sums about $d$ explicit: a pair of values
$a,b$ is matched with the pair $d-a,d-b$, and the two sums add to $2d$.

The only pairs fixed at the central cutoff are those with sum exactly $d$. Their number is
\(\sum_a n_an_{d-a}=\sum_a n_a^2.\)
The remaining
\(\mu^2-\sum_a n_a^2\)
ordered pairs therefore occur in equal numbers below and above $d$. By
\cref{cor:mk-count},
\(m_{d-1}(h\boxplusop h)=H_{h\boxplusop h}(d)\)
counts the self-sum spectral exponents strictly below $d$; equivalently, by definition it is
\[
 \dim_{\C}
 \frac{R_{h\boxplusop h}}
 {I_{d-1}(Z_{h\boxplusop h})+(h\boxplusop h,J_{h\boxplusop h})}.
\]
Because the cutoff is strict, none of the $\sum_an_a^2$ central pairs are counted. Hence
exactly half of the remaining pairs contribute, giving
\(m_{d-1}(h\boxplusop h) =\frac{\mu^2-\sum_an_a^2}{2}.\)
Finally,
\(\sum_an_a^2\ge\sum_an_a=\mu,\)
so
\(m_{d-1}(h\boxplusop h)\le\frac{\mu(\mu-1)}2.\)
Equality holds exactly when $n_a^2=n_a$ for every $a$, i.e. exactly when every nonzero
multiplicity is one and all spectral exponents are distinct.
\end{proof}

For both $x^2+y^{13}$ and $x^3+y^7$, the twelve spectral exponents are distinct, so the central self-sum value is $m_1=66$. The unit-bin upper bound is $108$, so spectral symmetry strictly improves the elementary integer-bin estimate.

\section{Nonrecoverability from integer Hodge--Tjurina data}\label{sec:nonrecover}

The preceding bounds use only the integer sequence $\{m_k\}$, whereas exact Thom--Sebastiani counts also depend on the fractional spectral positions. The family below shows that the missing information cannot in general be recovered from the integer sequence. We write $h\boxplusop h$ for the self sum in two disjoint variable sets.

\begin{example}\label{ex:unbounded}
For every integer $r\ge2$, put
\(f_r=x^2+y^{6r+1}\) and \(g_r=x^3+y^{3r+1}\).
Both germs are isolated plane curve singularities and
\begin{equation}\label{eq:same-sequence}
 m_0(f_r)=m_0(g_r)=3r,
 \qquad
 m_k(f_r)=m_k(g_r)=6r\quad(k\ge1).
\end{equation}
For their self sums in disjoint variables,
\begin{equation}\label{eq:selfsum-discrepancy}
 m_0(f_r\boxplusop f_r)=0,
 \qquad
 m_0(g_r\boxplusop g_r)=\frac{r(r-1)}2.
\end{equation}
Accordingly, for every integer $B>0$ one can choose two such germs with identical integer Hodge--Tjurina sequences but with zeroth self-sum numbers differing by more than $B$.
\end{example}
\begin{proof}
For a Brieskorn--Pham plane curve $x^a+y^b$, the spectral exponents in our convention are
\(\alpha_{i,j}=\frac ia+\frac jb\) for \(1\le i<a\) and \(1\le j<b\).
In the two families under consideration,
\(\gcd(2,6r+1)=\gcd(3,3r+1)=1.\)
Hence both germs are reduced and isolated, and
\(\mu(f_r)=(2-1)(6r+1-1)=6r\) and \(\mu(g_r)=(3-1)(3r+1-1)=6r\).
No spectral exponent can equal $1$. Indeed, if
\(\frac ia+\frac jb=1,\)
then $bi+aj=ab$, so reducing modulo $a$ gives $bi\equiv0\pmod a$. Since
$\gcd(a,b)=1$, one obtains $a\mid i$, which is impossible for $1\le i<a$.

The involution
\((i,j)\longmapsto(a-i,b-j)\)
sends
\(\alpha_{i,j}\longmapsto \alpha_{a-i,b-j}=2-\alpha_{i,j}.\)
Thus it pairs exponents below $1$ with exponents above $1$. Since no exponent equals
$1$, exactly half of the $6r$ exponents lie below $1$, namely $3r$, while all $6r$ lie
below $2$. By \cref{cor:mk-count},
\(m_0(f_r)=m_0(g_r)=3r\) and \(m_k(f_r)=m_k(g_r)=6r\) for \(k\ge1\),
which proves \eqref{eq:same-sequence}.

For later comparison, it is useful to write the spectrum of $f_r$ explicitly:
\(\alpha_j=\frac12+\frac{j}{6r+1}\) for \(1\le j\le6r\).
Its minimal spectral exponent is therefore
\(\frac12+\frac1{6r+1}.\)
Consequently the smallest spectral exponent of the self sum $f_r\boxplusop f_r$ is
\(1+\frac2{6r+1}>1.\)
Hence no self-sum exponent is below $1$, and therefore
\(m_0(f_r\boxplusop f_r)=0.\)
The same conclusion follows directly from the weighted spectral-level description in
\cref{rem:rho}: the constant monomial already has level
\(\frac12+\frac1{6r+1}+\frac12+\frac1{6r+1} =1+\frac2{6r+1}>1.\)

Now consider
\(g_r\boxplusop g_r=x^3+y^{3r+1}+u^3+v^{3r+1}.\)
A Milnor-algebra basis is
\[
 x^{i-1}y^{j-1}u^{a-1}v^{b-1},
 \qquad 1\le i,a\le2,\quad1\le j,b\le3r.
\]
The weighted spectral level of such a basis monomial is
\(\frac{i+a}{3}+\frac{j+b}{3r+1}.\)
To be strictly below $1$, the first term must be $2/3$, so necessarily $i=a=1$.
The remaining condition is
\[
 \frac23+\frac{j+b}{3r+1}<1
 \iff
 \frac{j+b}{3r+1}<\frac13
 \iff
 j+b<r+\frac13
 \iff
 j+b\le r,
\]
where the last equivalence uses that $j+b$ is an integer. Thus the desired basis monomials
are parametrized by the lattice points
\(j\ge1\), \(b\ge1\), and \(j+b\le r\).
For a fixed sum $s=j+b$, with $2\le s\le r$, there are exactly
\((1,s-1),(2,s-2),\ldots,(s-1,1),\)
namely $s-1$ ordered positive pairs. The original bounds $j,b\le3r$ are automatic on this
triangle. Therefore
\[
 m_0(g_r\boxplusop g_r)
 =\sum_{s=2}^r(s-1)
 =1+2+\cdots+(r-1)
 =\frac{r(r-1)}2.
\]
This proves \eqref{eq:selfsum-discrepancy}; its right-hand side tends to infinity with $r$.
\end{proof}

\paragraph{The first member.}
For $r=2$, the family gives $f=x^2+y^{13}$ and $g=x^3+y^7$, with common integer
sequence $m_0=6$ and $m_k=12$ for $k\ge1$, but self-sum values $0$ and $1$ at $m_0$.
\cref{prop:central-cutoff} gives in both cases $m_1=66$, so the discrepancy
occurs strictly below the central cutoff.

\begin{remark}
The unboundedness in \cref{ex:unbounded} occurs along the sequence of Milnor numbers $\mu=6r$. It does not assert unbounded discrepancy at a fixed Milnor number.
\end{remark}

\noindent\textbf{Declaration on the Use of Al Tools.} \Cref{thm:filtered-defect,prop:rational-recovery,prop:refined-grid,ex:unbounded} were found with AI assistance. AI tools were used for language editing, including grammar, syntax, and stylistic refinement, and for limited assistance with presentation consistency. The AI tool used was ChatGPT. All AI-generated content has been checked by the authors, who take full responsibility for the final manuscript.

\end{document}